\documentclass[12pt,leqno]{article}
\usepackage[active]{srcltx}
\usepackage{amsmath, amsthm, amsfonts, amssymb, color}
\usepackage[matrix,arrow]{xy}
\usepackage{mathrsfs}
\usepackage{enumerate}
\usepackage{amsopn} 
\usepackage{bbm} 
\newenvironment{packed_enum}{
\vspace{-3mm}
\begin{enumerate}
}{\end{enumerate}\vspace{-5mm}}

\newcommand{\E}{\mathbb{E}}

\newcounter{cprop}[section]

\newcommand{\mcB}{\mathcal{B}}

\newtheorem{definition}[cprop]{Definition}
\newtheorem{remark}[cprop]{Remark}
\newtheorem{lemma}[cprop]{Lemma}
\newtheorem{example}[cprop]{Example}
\newtheorem{proposition}[cprop]{Proposition}
\newtheorem{theorem}[cprop]{Theorem}

\usepackage{mdef}

\title{\bf Strong Solutions for Stochastic Partial Differential Equations of Gradient Type}

\author{
\textbf{Benjamin Gess}
\thanks{Supported by DFG-Internationales Graduiertenkolleg “Stochastics and Real World Models”, the SFB-701 and the BiBoS-Research Center. \vskip5pt \textbf{Acknowledgments:} The author would like to thank Professor Michael R\"ockner for valuable discussions and comments.} \\
\small{Faculty of Mathematics, University of Bielefeld, Germany} \\
\small{bgess@math.uni-bielefeld.de}
}

\begin{document}
 
\maketitle\vskip-20pt

\begin{abstract}
 Unique existence of analytically strong solutions to stochastic partial differential equations (SPDE) with drift given by the subdifferential of a quasi-convex function and with general multiplicative noise is proven. The proof applies a genuinely new method of weighted Galerkin approximations based on the ``distance`` defined by the quasi-convex function. Spatial regularization of the initial condition analogous to the deterministic case is obtained. The results yield a unified framework which is applied to stochastic generalized porous media equations, stochastic generalized reaction diffusion equations and stochastic generalized degenerated $p$-Laplace equations. In particular, higher regularity for solutions of such SPDE is obtained.
\end{abstract}

\noindent {\bf 2000 Mathematics Subject Classification AMS}: Primary: 35A15, 60H15. Secondary: 46N10, 76S05, 35J92, 35K57


\bigskip
\maketitle


\section{Introduction}

We prove higher spatial regularity of solutions to SPDE and thereby obtain the unique existence of strong (in the analytic sense) solutions. In general, the regularity of solutions to SPDE is worse than that of their deterministic counterparts. Therefore, instead of considering (analytically) strong solutions, usually weaker notions of solutions, like mild (see e.g. \cite{DPZ92,PZ07} and the references therein) or variational solutions (see e.g. \cite{KR79,PR07} and the references therein) are used. In this paper, we show that for SPDE with drift given by the subgradient of a quasi-convex function and with sufficiently regular noise, i.e. equations of the form
\begin{equation}\label{eqn:SPDE_intro}
  dX_t = -\partial\vp(X_t)dt + B_t(X_t)dW_t,
\end{equation}
where $\vp: H \to \bar\R$ is a quasi-convex function on a Hilbert space $H$, it is in fact possible to prove more regularity for the solutions and thereby to obtain (analytically) strong solutions.

The method we develop and employ to prove the regularity of solutions seems to be new even in the deterministic case. As a special case our results fully contain the analogous results for the deterministic case, in which it is well known that flows generated by gradient type PDE have regularizing properties similar to analytic semigroups (cf.\ e.g.\ \cite{B93}). In this sense the regularity properties obtained here are optimal.

We point out that our results yield a unified framework for semilinear equations like stochastic reaction diffusion equations (SRDE) as well as for quasilinear equations like stochastic porous media equations (SPME). In contrast to our framework, truly quasilinear equations like SPME cannot be treated in the mild approach, while SRDE with high order growth of the nonlinearity do not fit into the variational framework. \\
In the case that both the method established here and the variational approach as presented in \cite{RRW07} apply, our results show that variational solutions are in fact strong solutions.

Applications of our results to stochastic generalized porous media equations, stochastic generalized reaction diffusion equations and stochastic generalized $p$-Laplace equations are given in detail in Section \ref{sec:applications}, also including the case of a not necessarily bounded underlying domain $\mcO \subseteq \R^d$, and covering different types of boundary conditions.
 
We now comment on existing regularity results for SPDE and compare to the results established in this paper. Concerning the regularity of solutions to linear SPDE a lot of work has been done (cf.\ e.g.\ \cite{F90,F95,R90,RS81,K94,DPZ92,BvN03,BvNVW08}), but much less is known for the nonlinear case. Applied to linear SPDE, as e.g.
  \[ dX_t = (\D X_t + f(t)) dt + B_t dW_t \]
our results recover regularity properties obtained in \cite[Theorem 3.1]{F90} under the same conditions as used in our paper, i.e. $X_0 \in H_0^1(\mcO)$ and $B_t(\td e_k) \in L^2([0,T]\times\O;H_0^1(\mcO))$. Note that these conditions are necessary at least for more regular solutions (\cite[Proposition 3.3]{F90}). \\
In \cite{JR10} semilinear equations with non-linear diffusion and Lipschitz nonlinear drift of the form
  \[ dX_t = (AX_t + F(X_t))dt + B_t(X_t)dW_t \]
are considered and under appropriate assumptions on the generator $A$ and the diffusion part $B$ it is shown that for initial values $X_0 \in \mcD(A^\g)$ the solution $X$ also takes values in $ \mcD(A^\g)$, requiring $\g \in [0,\frac{3}{4}]$. By use of the $L^p$-theory for SPDE in \cite{Z07} higher regularity for solutions to semilinear SPDE with nonlinear diffusion and regularization with respect to the initial value by any order less than $\frac{1}{2}$ of $A$ is obtained. Assuming that the drift $A+F$ is a subdifferential of a quasi-convex function and under appropriate assumptions on the diffusion $B$ (in particular trace class noise) these results are partially extended here by proving regularization of the initial condition by one order of $A$ (i.e. $X \in \mcD(A) \subseteq H$ a.s.). Note that in the linear case $A=-\partial\vp$ implies that $A$ is an analytic operator.\\
The only reference known to the author concerning the regularity of solutions to general quasilinear SPDE is \cite{L10}. For nonlinear SPDE fitting into the variational framework invariance of certain subspaces $S$ of the state space $H$ is shown. Roughly speaking, this means that if the initial condition $X_0$ takes values in the subspace $S$ then so does the solution $X$. The proof is based on a Yosida approximation of the norm $\|\cdot\|_H$, which in the special case of stochastic porous media equations (SPME) 
  \[dX_t = \D (|X_t|^{p-1}X_t) dt + B_t(X_t)dW_t,\ p \ge 2 \]
had been used in \cite{RW08} to solve the SPME on $L^2$ satisfying $|X|^\frac{p-1}{2}X \in L^2([0,T]\times\O;H_0^1)$. For the SPME with additive noise the same had been obtained in \cite{K06}. Concerning the SPME with additive noise the strongest regularity result so far is given in \cite[Lemma 3.3]{BGLR10} proving $|X|^{p-1}X \in L^{\frac{p+1}{p}}([0,T];H_0^{1,\frac{p+1}{p}})$. Our results are applicable to general quasilinear SPDE perturbed by general multiplicative noise and yield the unique existence of strong solutions, in particular $A(X) \in L^2([0,T]\times\O;H)$, $H=(H_0^1)^*$) and spatial regularization of the initial condition in the sense that for $X_0 \in L^2(\O;H)$ we obtain a solution $X$ satisfying $ t^\frac{1}{2} A(X_t) \in L^2([0,T]\times\O;H)$. In the special case of the SPME the regularity results in \cite{L10,RW08,K06,BGLR10} mentioned above are thus strengthened to 
  \[ t^\frac{1}{2}|X_t|^{p-1}X_t \in L^2([0,T]\times\O;H_0^1(\mcO)),\]
under appropriate assumptions on the noise. For more details on existing regularity results for solutions to concrete SPDE and comparison to the results obtained here we refer to Section \ref{sec:applications} below.

We now give a brief outline of the main idea of the proof. The regularity analysis of solutions to deterministic PDE of gradient type (cf.\ e.g.\ \cite{B93}, also for notation) is based on a chain rule for absolutely continuous functions $u \in W^{1,2}([0,T];H)$ composed with a convex, lower semicontinuous (l.s.c.) function $\vp: H \to \bar \R$ with $u(t) \in \mcD(\partial \vp)$ for almost every $t\in [0,T]$, stating that $\vp \circ u$ is absolutely continuous and 
\begin{equation}\label{eqn:det_chain}
   \frac{d}{dt} \vp \circ u = \left( g(t),\frac{du}{dt}(t) \right)_H,
\end{equation}  
were $g(t) \in \partial\vp(u(t))$ almost surely. The main difficulty in the stochastic case thus stems from the lack of a sufficiently general It\^o formula justifying \eqref{eqn:det_chain} in an integral sense for SPDE. The deterministic chain rule mentioned above is easily proven by a Moreau-Yosida approximation $\vp_\l$ of the convex function $\vp$ and by applying the classical chain rule. In the stochastic case, however, this method fails due to the occurrence of the second derivative $D^2\vp_\l$ of the Moreau-Yosida approximation of $\vp$ in It\^o's formula. To proceed further, uniform bounds on $D^2\vp_\l$ would have to be shown, which seems to be difficult. \\
Therefore, we establish an alternative method of approximation which is not based on a regularization of $\vp$ so that only bounds on $D^2\vp$ (not on its approximation $D^2\vp_\l$) are needed. For an orthonormal basis $e_k$ of $H$ the usual Galerkin approach is to project the equation from $H$ to a finite dimensional equation on $H_n := $ span $\{e_1,...,e_n\}$ by applying the orthogonal projection $P_n: H \to H_n$. The key point of this approximation is that the orthogonal projection $P_n$ preserves the monotonicity of $\partial\vp$ on $H$. However, this projection is only based on the geometry of $H$ and does not fully reflect the geometry of the equation which is induced by the convex function $\vp$. For example, in general $P_n$ is not bounded with respect to $\vp$, i.e.\ $\vp(P_n h) \not\le C \vp(h),\ \forall h \in H$ and the projection does not converge to the identity with respect to $\vp$, i.e.\ $\vp(P_n h - h) \not\to 0$, for $n \to \infty$. Therefore, for the approximation of the initial condition and the noise we instead apply the best approximation $\mcP_{\vp,n}: H \to H_n$ weighted by $\vp$, i.e.
  \[ \vp(\mcP_{\vp,n} h -h) = \inf_{g \in H_n} \vp(g-h). \]
The point is that this approximation respects $\vp$ in the sense that
  \[ \vp(\mcP_{\vp,n} h) \le C \vp(h),\ \forall h \in H \]
and the convergence $\mcP_{\vp,n} h \to h$ holds with respect to $\vp$, i.e.
\begin{align*}
  \vp(\mcP_{\vp,n} h - h) &\to 0, \\
  \vp(\mcP_{\vp,n} h)     &\to \vp(h),\ \forall h \in H.
\end{align*}

In Section \ref{sec:main_result} we will establish the framework and state the main results. If the initial condition is contained in the domain of the quasi-convex function $\vp$, Theorem \ref{thm:main} asserts the existence of a unique strong solution to \eqref{eqn:SPDE_intro} and provides a-priori bounds for this solution. The proof is given in Section \ref{sec:proof_main_thm}. If the initial condition is not contained in the domain of $\vp$ but only $X_0 \in L^2(\O,\mcF_0;H)$ we first prove the unique existence of solutions in a limiting sense (Theorem \ref{thm:limit_soln}) and then identify these limit solutions as strong solutions to \eqref{eqn:SPDE_intro} in some generalized sense (Theorem \ref{thm:regularity}). In particular, regularization of the initial condition is shown. The proofs are given in Section \ref{sec:limit_soln}. The application of our results to SPME, SRDE and stochastic $p$-Laplace equations is given in Section \ref{sec:applications}.

\section{Setup and Main Results}\label{sec:main_result}
Let $H$ be a separable Hilbert space, $V$ be a reflexive Banach space and $S$ be a Banach space (note that not necessarily $S \subseteq H$, nor $S$ reflexive). Assume that there are dense embeddings $V \hookrightarrow S, H$ and a constant $C > 0$ such that $\|\cdot\|_V \le C(\|\cdot\|_S+\|\cdot\|_H).$ In the applications (cf. Section \ref{sec:applications}) we always have $V = $``$S \cap H$'' in some appropriate sense. By reflexivity of $V$ this yields the Gelfand triple $V \subseteq H \subseteq V^*$. In particular, $V$ is separable and we can choose $\{e_k \in V|\ k \in \N\}$ to be an orthonormal basis of $H$ such that $V_0 = $ span$\{e_1, \cdots\}$ is dense in $V$. Define $H_n := $ span$\{e_1,\cdots,e_n\}$. Let $U$ be another separable Hilbert space, $(\O,\F,\F_t,\P)$ be a complete probability space with normal filtration $\F_t$, $W_t$ be a cylindrical Brownian motion on $U$ and $B: [0,T] \times \O \times H \to L_2(U,H)$ be progressively measurable.

We now introduce the main assumptions under which we will prove the unique existence of strong solutions to SPDE of the form \eqref{eqn:SPDE_intro}. Although some of the conditions might look unfamiliar at the first sight, all of them are easy to prove in the applications (Section \ref{sec:applications}).

\begin{packed_enum}
 \item[(A1)] 
   Let $\vp: S \to \R$ be a continuous function. Assume that there exists a convex, continuous, subhomogeneous (i.e.\ $\exists\ C > 0$ such that $\td\vp(2x) \le C \td\vp(x)$ for all $x \in S$) function $\td\vp: S \to \R_+$ with bounded level-sets such that $\td\vp(v_n) \to 0$ implies $v_n \to 0$ in $S$ and that there are constants $C_1 > 0$, $C_2 \in \R$ so that
   \begin{align*}
        C_1 \td\vp(v)- C_2 \le \vp(v) &\le C_2 (\td\vp(v)+\frac{1}{2}\|v\|_H^2+1),
    \end{align*}
    for all $v \in V$.
\end{packed_enum}  

Via the embeddings $V \hookrightarrow S, H$ we may define $\vp_\l(v) := \vp(v)+\frac{\l}{2}\|v\|_H^2$ and $\td\vp_\l(v) := \td\vp(v)+\frac{\l}{2}\|v\|_H^2$, for $v \in V$, $\l \in \R$. 

\begin{packed_enum}
  \item[(A2)] Assume that $\vp$ is Gateaux differentiable on $V$ with $A := -D\vp: V \to V^*$ being hemicontinuous and that $\vp \in C^2(H_n)$ for all $n \in \N$, such that there are constants $1 = p_0 \le p_2 \le ... \le p_N$, $N \in \N$ for which
      \[ \sum_{k=1}^\infty D^2\vp(x)(w_k,w_k) \le C \left(1+\vp(x)+\sum_{i=0}^N\left(\sum_{k=1}^\infty \td\vp_1(w_k)^\frac{1}{p_i} \right)^{p_i} \right), \]
       for each sequence $w_k \in H_n$, all $x \in H_n$ and some constant $C \in \R$ (independent of $n$).
\end{packed_enum}  

We can consider $\vp,\td\vp,\td\vp_\l,\vp_\l$ as functions on $H$ by extending them by $+\infty$ from $V$ to $H$.

\begin{packed_enum}
\item[(A3)]  
 Assume that  $\vp: H \to \bar\R$ is quasi-convex (i.e. $\exists \l \ge 0$ such that $\vp_\l$ is convex). 
\end{packed_enum}  

The notion of the subgradient is extended to quasi-convex functions in the obvious way. In particular, we have $\partial \vp =  \partial \vp_\l - \l x$. We extend $\partial \vp: \mcD(\partial\vp) \subseteq V \to H$ by $0$ to all of $H$. Note that $\mcD(\partial\vp) = \{v \in V|\ D\vp(v) \in H\}$ and $\partial\vp = D\vp_{|\mcD(\partial\vp)}$ (in particular $\partial\vp(v)$ is single-valued for all $v \in \mcD(\partial\vp)$). Boundedness of level-sets of $\td\vp: V \to \R$, reflexivity of $V$ and weak l.s.c.\ of $\vp_\l$ imply
\begin{remark}
  Assume (A1), (A3). Then $\vp, \td\vp: H \to \bar\R $ are lower semicontinuous.
\end{remark}

\begin{packed_enum}
  \item[(A4)] (Weak coercivity): There exists a $C \in \R$ such that
    \[ 2 \Vbk{-D \vp (v),v} \le C (1+\|v\|_H^2),\] 
  for all $v \in V$.
 \item[(A5)] (Lipschitz noise):  There is a constant $c \ge 0$ such that
        \[ \|B_t(v)-B_t(w)\|_{L_2(U,H)}^2 \le c \|v-w\|_H^2,\]
        for all $v,w \in V$.
 \item[(A6)] (Regularity of the noise): There is an orthonormal basis $\td e_k$ of $U$, such that
        \[ \|B_t(v)\|_{L_{2,\td\vp_1},(p_i)} := \sum_{i=0}^N \left( \sum_{k=1}^\infty \td\vp_1(B_t(v)(\tilde{e_k}))^\frac{1}{p_i} \right)^{p_i} \le C(f_t+\vp(v)+ \|v\|_H^2), \]
        where $p_i$ are as in (A2), $C > 0$ is a constant and $f_t \in L^1([0,T]\times\O)$ is an $\mcF_t$-adapted process.
\end{packed_enum}

As mentioned above all these conditions are easily verified in the applications and pose a unified framework for quasilinear equations like porous media equations and reaction diffusion equations with high order reaction terms. Note that we neither require strict coercivity of $\vp$, nor polynomially bounded growth of $\vp$. Since 
\begin{align*}
  \Vbk{-D\vp(v),v} &= -\Vbk{ D\vp_\l(v)-D\vp_\l(0),v-0} - \Vbk{D\vp_\l(0),v} + \l \|v\|_H^2 \\
		   &\le - \Vbk{D\vp_\l(0),v} + \l \|v\|_H^2
\end{align*}
we have 
\begin{remark}[Weak coercivity]
   Assume (A3) and $0 \in \mcD(\partial\vp)$. Then (A4) is satisfied.  
\end{remark}

In the following we will consider SPDE of the form
\begin{equation}\label{eqn:SPDE}
  dX_t = - \partial\vp(X_t) dt + B_t(X_t) dW_t,\text{ on } [0,T].
\end{equation}
Let us first state what we mean by a strong solution.

\begin{definition}[Strong solution]
  Let $X_0 \in L^2(\O,\mcF_0;H)$. An $\F_t$-adapted, continuous process $X$ in $H$ with $X \in L^2(\O;C([0,T];H))$ and $\partial\vp(X) \in L^2([0,T]\times\O;H)$ is a strong solution to \eqref{eqn:SPDE} if $\P$-a.s. 
  \begin{equation}\label{eqn:SPDE_strong}
      X_t = X_0 - \int_0^t \partial\vp(X_r) dr + \int_0^t B_r(X_r) dW_r,\ \forall t \in [0,T],
  \end{equation}
  as an equation in $H$. 
\end{definition}

\begin{theorem}\label{thm:main} 
  Assume (A1)-(A6). Then for each initial condition $X_0 \in L^2(\O,\mcF_0;H)$ satisfying $\E\left(\vp(X_0)+\|X_0\|_H^2 \right) < \infty$ there is a unique strong solution $X$ to \eqref{eqn:SPDE} with
    \[ \vp(X)+\|X\|_H^2 \in L^\infty([0,T];L^1(\O)).\]
  Let $X^{(i)}_0 \in L^2(\O,\mcF_0;H)$ and let $X^{(i)}$ be corresponding strong solutions to \eqref{eqn:SPDE}, $i=1,2$. Then
  \begin{align}\label{eqn:dep_ic}
     \E \sup_{t\in[0,T]}\|X^{(1)}_t-X^{(2)}_t\|_H^2 \le C \E\|X^{(1)}_0-X^{(2)}_0\|_H^2.
  \end{align}
\end{theorem}

So far we have obtained strong solutions for initial conditions taking values in $V \subseteq H$. By monotonicity it is possible to obtain solutions to \eqref{eqn:SPDE} for arbitrary initial conditions $X_0 \in L^2(\O,\mcF_0;H)$ at least in a limit sense. 

\begin{definition}
  Let $X_0 \in L^2(\O,\mcF_0;H)$. We call an $H$-valued, $\F_t$-adapted, continuous process $X \in L^2(\O;C([0,T];H))$ a limit solution to \eqref{eqn:SPDE} if $X(0) = X_0$ and if there exists an approximating sequence $X^n \in L^2(\O;C([0,T];H))$ of strong solutions to \eqref{eqn:SPDE} (with initial conditions $X_0^n=X^n(0)$) such that $X^n \to X  \text{ in } L^2(\O;C([0,T];H)).$
\end{definition}

\begin{theorem}\label{thm:limit_soln}
  Let (A1)-(A6) hold and $X_0 \in L^2(\O,\mcF_0;H)$. Then there exists a unique limit solution $X \in L^2(\O;C([0,T];H))$ to \eqref{eqn:SPDE}. For two limiting solutions $X^{(i)}$ with corresponding initial conditions $X^{(i)}_0 \in L^2(\O,\mcF_0;H)$, $i=1,2$,
  \begin{equation}\label{eqn:limit_soln_ctn}
    \E \sup_{t \in [0,T]} \|X^{(1)}_t-X^{(2)}_t\|_H^2 \le C \E\|X^{(1)}_0-X^{(2)}_0\|_H^2,
  \end{equation}
  for some constant $C \ge 0$.
\end{theorem}

Note that we did not yet identify limit solutions in terms of some generalized notion of solution to \eqref{eqn:SPDE}. Moreover, uniqueness of limit solutions is only proven with respect to the specific choice of approximation via strong solutions. Other approximation methods could yield solutions in a limit sense which might be different from those obtained in Theorem \ref{thm:limit_soln}. We thus aim to introduce a generalized notion of a strong solution to \eqref{eqn:SPDE} and to prove that the limit solutions are in fact solutions to \eqref{eqn:SPDE} in this generalized sense. 

For comparison and to fix ideas let us recall the regularizing properties of deterministic gradient flows. Let $\vp: H \to \bar\R$ be a convex, l.s.c.\ function. For deterministic gradient flows in Hilbert spaces it is well known (cf.\ e.g.\ \cite[Theorem 1.10.]{B93}) that even for initial conditions $y_0 \in \overline{\mcD(\vp)}$ (and $f \in L^2([0,T];H)$) there exists a unique mild solution $y \in C([0,T];H)$ to
\begin{align*}
  \frac{dy}{dt} + \partial\vp(y) &= f(t), \text{ in } (0,T) \\
  y(0) &= y_0,
\end{align*}
such that for each $0<\d <T$, $y \in W^{1,2}([\d,T];H)$, $t^\frac{1}{2}\frac{dy}{dt} \in L^2([0,T];H)$, $\vp(y) \in L^1([0,T])$, $y(t) \in \mcD(\partial\vp)$ for a.e. $t \in (0,T)$ and
\begin{align}\label{eqn:determ_reg}
   \frac{dy}{dt}(t) + \partial\vp(y(t)) = f(t), \text{ a.e. } t \in (0,T).
\end{align} 

Since not necessarily $\frac{dy}{dt} \in L^1([0,T];H)$ we cannot pass to the integrated form of \eqref{eqn:determ_reg} over all of $[0,T]$ and thus we cannot expect an analytically strong solution in the stochastic case. However, since $\frac{dy}{dt} \in L^1([\d,T];H)$ for each $\d > 0$, we can expect the limit solution to be a strong solution on each interval $[\d,T]$. Using this property as a definition of a generalized strong solution, we may identify limit solutions as solutions to the SPDE \eqref{eqn:SPDE} and thereby prove uniqueness of limit solutions independent of the approximation. Moreover, we prove that limit solutions enjoy regularizing properties analogous to the deterministic case.

\begin{definition}
  Let $X_0 \in L^2(\O,\mcF_0;H)$. An $H$-valued, $\F_t$-adapted, continuous process $X \in L^2(\O;C([0,T];H))$ is said to be a strong solution to \eqref{eqn:SPDE} on $(0,T]$ if $X(0)=X_0$, $\partial\vp(X) \in L^2([\d,T]\times\O;H)$ and $\P$-a.s.
  \begin{equation}\label{eqn:SPDE_start_d}
    X_t = X_\d - \int_\d^t \partial\vp(X_r)dr + \int_\d^t B_r(X_r)dW_r,
  \end{equation}
  for all $t \in [\d,T]$ and all $0 < \d < T$.
\end{definition}

In order to prove regularity for the limit solution we need to require the following stronger coercivity property:
\begin{packed_enum}
 \item[(A4')] There exist constants $C_1 > 0$, $C_2 \in \R$ such that
  \[ 2 \Vbk{- D \vp(v),v } \le C_2(1+\|v\|_H^2) - C_1\vp(v), \]
  for all $v \in V$.
\end{packed_enum}

\begin{theorem}\label{thm:regularity}
  Assume (A1)-(A6), (A4'). Let $X_0 \in L^2(\O,\mcF_0;H)$ and $X \in L^2(\O;C([0,T];H))$ be the corresponding limit solution. Then
  \begin{align*}
    \vp(X) + \|X\|_H^2             &\in L^1([0,T] \times \O), \\
    t^\frac{1}{2} \partial\vp(X_t) &\in L^2([0,T]\times\O;H)
  \end{align*}
  and $X$ is the unique strong solution to \eqref{eqn:SPDE} on $(0,T]$.
\end{theorem}

If the derivative $A= -D\vp : V \to V^*$ satisfies additional coercivity and growth conditions (cf.\ (K), (H1)-(H4) in \cite{RRW07}) the variational approach to SPDE can be used to prove the unique existence of variational solutions to \eqref{eqn:SPDE}. In this case limit solutions and variational solutions coincide and consequently Theorem \ref{thm:regularity} provides a regularity result for variational solutions.
\begin{remark}\label{rmk:limit=varn}
  Assume (K), (H1)-(H4) from \cite{RRW07} and let $X^{(l)}$, $X^{(v)}$ be limit (resp. variational) solutions to \eqref{eqn:SPDE}. Then $X^{(l)}=X^{(v)}$ up to indistinguishability.
\end{remark}

\section{Proof of Theorem \ref{thm:main}}\label{sec:proof_main_thm}

Note that (A1) implies that $\vp$ is bounded from below. Therefore, without loss of generality we can assume $\vp \ge 0$ in all proofs. In the following $C$ denotes a constant which may change from line to line. We first consider the case of additive, time-dependent noise.

\subsection{Additive noise}

\begin{lemma}[Additive noise]\label{lemma:main_add}
  Assume (A1)-(A6), $X_0 \in L^2(\O,\mcF_0;H)$ with $\E \vp_1(X_0) < \infty$ and that $B_t(v) \equiv B_t$ is independent of $v \in V$. Then there is a unique strong solution $X$ to 
  \begin{equation}\label{eqn:add_SPDE}
      dX_t = - \partial\vp(X_t) dt + B_t dW_t,\text{ on } [0,T]
  \end{equation}
  satisfying
  \begin{align}\label{eqn:add_energy_bound}
    &\sup_{t \in [0,T]} \E (\vp(X_t)+\|X_t\|_H^2) + \E \int_0^T \|\partial\vp(X_r)\|_H^2 dr \\
    &\hskip100pt \le  C \left( \E \vp_1(X_0) + 1 + \E \int_0^t \|B_r\|_{L_{2,\td\vp_1},(p_i)} dr \right), \nonumber
  \end{align}
  for some $C > 0$. Let $X^{(i)}_0 \in L^2(\O,\mcF_0;H)$ with $\E \vp_1(X^{(i)}_0) < \infty$, $B^{(i)}: [0,T] \times \O \to L_2(U,H)$ be two progressively measurable processes satisfying (A5), (A6) and $X^{(i)}$ be the corresponding strong solutions to \eqref{eqn:add_SPDE}. Then
  \begin{equation}\label{eqn:dep_ic_noise_add}
    \E \sup_{t \in [0,T]} \|X^{(1)}_t-X^{(2)}_t\|_H^2 \le C \E \left( \|X^{(1)}_0-X^{(2)}_0\|_H^2 + \int_0^T \|B_r^{(1)}-B_r^{(2)}\|_{L_2(U,H)}^2 dr \right). 
  \end{equation}
\end{lemma}
In order to prove the existence of solutions to \eqref{eqn:add_SPDE} we consider finite dimensional approximations. Let $H_n := \text{span}\{e_1,\dots,e_n\} \subseteq V \subseteq S,H$ and let $P_n: V^* \to H_n$ be defined by $P_n h := \sum_{k=1}^n e_k \Vbk{h,e_k}$. For $\l >0$, $\td\vp_\l$ is a continuous, subhomogeneous function with bounded level-sets and $\td\vp_\l(v_n) \to 0$ implies $v_n \to 0$ in $V$. Moreover, $\td\vp_1$ is strictly convex and 
\begin{equation}
  C_1\td\vp(v)-C_2 \le \vp(v) \le C_2(\td\vp_1(v)+1),
\end{equation}
for all $v \in V$. Since $H_n$ is convex and locally compact, the $\td\vp_1$-best approximation map $\mcP_{\vp,n}: V \to H_n$ is well-defined by
  \[ \td\vp_1(v-\mcP_{\vp,n}v) = \inf_{h \in H_n} \td\vp_1(v-h) \]
and is continuous (cf. \cite[Theorem 5, 6, 8]{GS82}). Since $\td\vp_1$ is subhomogeneous and $0 \in H_n$, there is a constant $C > 0$ (independent of $n$) such that
\begin{equation}\label{eqn:approx_bound}
  \td\vp_1(\mcP_{\vp,n}v) \le C \td\vp_1(v).
\end{equation}
Denote $X_0^n := \mcP_{\vp,n} X_0$, let $\{\tilde{e_k}|\ k\in \N\}$ be an orthonormal basis of $U$ as in (A6) and consider the approximating equation 
  \begin{align}\label{eqn:approx_eqn}
     X_t^n &= \mcP_{\vp,n} X_0 - \int_0^t P_{n} D\vp(X_r^n) dr + \sum_{k=1}^n \int_0^t \mcP_{\vp,n} B_r(\td e_k) d\b^k_r.
  \end{align}
By a standard existence result for finite dimensional monotone SDE there is a unique continuous strong solution to \eqref{eqn:approx_eqn} (cf. \cite[Theorem 3.1.1.]{PR07}). We now derive a-priori estimates for \eqref{eqn:approx_eqn}. 
\begin{lemma}\label{lem:a-priori-bdd}
  Assume (A1)-(A6), $X_0 \in L^2(\O,\mcF_0;H)$ satisfying $\E \vp_1(X_0) < \infty$ and that $B_t(v) \equiv B_t$ is independent of $v \in V$. Then there exists a constant $C \in \R$ such that     
  \begin{align}\label{eqn:a-priori-bdd}
    &\sup_{t \in [0,T]} \E \left( \vp(X_t^n)+\|X_t^n\|_H^2 \right) + \E \int_0^T \|P_nD \vp(X_r^n)\|_H^2 dr \\
    &\hskip100pt \le C \left( \E \vp_1(X_0) + 1 + \E \int_0^t \|B_r\|_{L_{2,\td\vp_1},(p_i)} dr \right) \nonumber.
  \end{align}
\end{lemma}
\begin{proof}
  Since $\vp \in C^2(H_n)$, we can apply It\^o's formula for the finite dimensional approximation. We get
  \begin{align*}
    \vp(X_t^n) 
        &= \vp(X_0^n) - \sum_{i=1}^n \int_0^t (P_n D \vp(X_r^n),e_i)_H^2 dr + \sum_{i=1}^n \sum_{k=1}^n \int_0^t D \vp(X_r^n)(e_i) (\mcP_{\vp,n}B_r(\td e_k),e_i)_H d\b^k_r \\
        &\hskip20pt + \frac{1}{2} \sum_{k=1}^n \int_0^t  D^2\vp(X_r^n)(\mcP_{\vp,n}B_r(\tilde{e_k}),\mcP_{\vp,n}B_r(\tilde{e_k})) dr,
  \end{align*}
  Using (A2) and \eqref{eqn:approx_bound} we observe:
  \begin{align*}
    &\frac{1}{2} \sum_{k=1}^n \int_0^t  D^2\vp(X_r^n)(\mcP_{\vp,n}B_r(\tilde{e_k}),\mcP_{\vp,n}B_r(\tilde{e_k})) dr \\
    &\hskip15pt \le \frac{1}{2} \int_0^t  C \left(1 + \vp(X_r^n) + \sum_{i=0}^N \left( \sum_{k=1}^\infty \td\vp_1(\mcP_{\vp,n}B_r(\tilde{e_k}))^\frac{1}{p_i} \right)^{p_i} \right) dr \\
    &\hskip15pt = C \int_0^t \left(1 + \vp(X_r^n) + \|B_r\|_{L_{2,\td\vp_1},(p_i)} \right) dr.
  \end{align*}

  By using a standard localization argument, (A6) and \eqref{eqn:approx_bound} we thus obtain
  \begin{align*}
    \E e^{-Ct} \vp(X_t^n) + \E \int_0^t e^{-Cr}\|P_n D \vp(X_r^n)\|_H^2 dr 
    &\le \E \vp(X_0^n) + C \E \int_0^t e^{-Cr} \left(1 + \|B_r\|_{L_{2,\td\vp_1},(p_i)} \right) dr \\
    &\le C \left( \E \td\vp_1(X_0) + 1 + \E \int_0^t \|B_r\|_{L_{2,\td\vp_1},(p_i)} dr \right).
  \end{align*}
  The remaining bound for $\|X_t^n\|_H^2$ follows by applying It\^o's formula to $\|\cdot\|_H^2$.
\end{proof}

\begin{proof}[Proof of Lemma \ref{lemma:main_add}.]\ 

By Lemma \ref{lem:a-priori-bdd} we know that $X^n$ and $P_n D\vp(X^n)$ are uniformly bounded in $L^2([0,T] \times \O;H)$.
By reflexivity of $L^2([0,T] \times \O;H)$ we can thus find a subsequence (again denoted by $n$), such that
\begin{packed_enum}
 \item $X^n \rightharpoonup \bar{X}$ in $L^2([0,T] \times \O;H )$.
 \item $P_n D\vp(X_r^n) \rightharpoonup \eta$ in $L^2([0,T] \times \O;H )$ and hence
       \[ \int_0^\cdot P_n D\vp(X_r^n) dr \rightharpoonup \int_0^\cdot \eta_r dr, \text{ in } L^2([0,T]\times\O;H).\]
\end{packed_enum}
Since $V_0 \subseteq V$ is dense, $\td\vp_1: V \to \R$ is continuous, subhomogeneous and $\td\vp_1(v_n) \to 0$ implies $v_n \to 0$ in $V$, we have 
\begin{remark}\label{rmk:vp1_dense}
  $\mcP_{\vp,n}(v) \xrightarrow{n \to \infty}  v$, for all $v \in V$.
\end{remark}
Using \eqref{eqn:approx_bound} and dominated convergence implies $\mcP_{\vp,n}B_t \to B_t$ in $L^2([0,T]\times\O;L_2(U,H))$ and $X_0^n = \mcP_{\vp,n}X_0 \to X_0$ in $L^2(\O;H)$. Hence
  \[ \sum_{k=1}^n \int_0^\cdot \mcP_{\vp,n}B_r(\td e_k) d\b^k_r  \to \int_0^\cdot B_r dW_r,\text{ in } L^\infty([0,T];L^2(\O;H)).\]

Let $\Phi \in L^2([0,T] \times \O;H)$. Then
\begin{align*}
  \E \int_0^T (\bar{X}_s,\Phi_s)_H ds
   &= \lim_{n \to \infty} \E \int_0^T (X^n_s,\Phi_s)_H ds \\
   &= \lim_{n \to \infty} \E \int_0^T \Big(X_0^n - \int_0^s P_n D \vp(X_r^n) dr + \sum_{k=1}^n \int_0^s \mcP_{\vp,n}B_r(\td e_k) d\b^k_r,\Phi_s \Big)_H ds \\
   &= \E \int_0^T \Big(X_0 - \int_0^s \eta_r dr + \int_0^s B_r dW_r,\Phi_s\Big)_H ds.
\end{align*}
Let now
\begin{equation}\label{eqn:X-def}
  X_t := X_0 - \int_0^t \eta_r dr + \int_0^t B_r dW_r.
\end{equation}
Then $\bar{X} = X$ dt$\otimes\P$-a.s.\ and $X_t$ is an $\mcF_t$-adapted, continuous process in $H$. 

It remains to identify the limit $\eta$, i.e.\ we need to prove $\eta = \partial\vp(X)$ or equivalently $\eta - \l X = \partial\vp_\l(X)$, almost surely. Let $t \in (0,T]$ and $\l > 0$ such that $\vp_\l$ is convex. We define a mapping $\Phi_{\l,t}: L^2([0,t]\times\O;H) \to \bar{\R}$ by
$$ \Phi_{\l,t}(Z) := 
\begin{cases}
  \E \int_0^t e^{- 2\l r }\vp_\l(Z_r)dr        & \text{, if } \vp_\l(Z) \in L^1([0,t]\times\O) \\
            + \infty                            & \text{, otherwise.}
\end{cases} $$
Since $\vp_\l$ is convex, l.s.c.\ and non-negative, $\Phi_{\l,t}$ is a convex, l.s.c.\ function. We will show $e^{-2\l \cdot}(\eta_\cdot - \l X_\cdot) \in \partial\Phi_{\l,T}(X)$. We know: 
  \[ \vp_\l(X^n_r) \le \vp_\l(v) + \Vbk{D \vp_\l(X^n_r),X^n_r-v} = \vp_\l(v) + \Vbk{P_n D \vp_\l(X^n_r),X^n_r} - \Vbk{D \vp_\l(X^n_r),v}, \]
for all $v \in V$, $r \in [0,T]$, $\o \in \O$. For $Z \in \mcD(\Phi_{\l,t})$ this implies
\begin{align}\label{eqn:subgr_1}
  \Phi_{\l,t}(X^n) 
  &= \Phi_{\l,t}(Z) + \E\int_0^t e^{-2\l r} (P_n D \vp_\l(X^n_r),X^n_r)_H dr \\
     &\hskip15pt -  \E\int_0^t e^{-2\l r} \Vbk{ D \vp_\l(X^n_r),Z_r } dr. \nonumber 
\end{align}
By It\^o's formula applied to $\|\cdot\|_H^2$ we get
\begin{align}\label{eqn:Ito_approx}
  &\E e^{-2\l t} \|X_t^n\|_H^2 \\
  &= \E \|X_0^n\|_H^2 + \E \int_0^t e^{-2 \l r} \left( 2(-P_n D \vp_\l(X_r^n),X_r^n)_H + \|\mcP_{\vp,n} B_r\|_{L_2(U,H)}^2 \right) dr.\nonumber 
\end{align}
Using this in \eqref{eqn:subgr_1}, multiplying with a non-negative function $\Psi \in L^\infty([0,T])$ and integrating yields
\begin{align}\label{eqn:subgr_2}
   &\int_0^T \Psi_t \Phi_{\l,t}(X^n) dt
   \le \int_0^T \Psi_t \Phi_{\l,t}(Z) dt \nonumber\\
     &\hskip75pt + \frac{1}{2} \int_0^T \Psi_t \bigg(  \E \|X_0^n\|_H^2 + \E \int_0^t e^{-2\l r} \|\mcP_{\vp,n} B_r \|_{L_2(U,H)}^2 dr - \E e^{-2\l t}\|X_t^n\|_H^2 \bigg) dt  \\
     &\hskip75pt - \int_0^T \Psi_t \E\int_0^t e^{-2\l r} \Vbk{D\vp_\l (X^n_r),Z_r} dr dt.\nonumber
\end{align}
We have seen above that $\E\|X_0^n\|_H^2 \to \E\|X_0\|_H^2$ and $\mcP_{\vp,n} B_r \to B_r$ in $L^2([0,T]\times\O;L_2(U,H))$. Since $X_t^n \rightharpoonup X$ in $L^2([0,T]\times\O;H)$ and $\Phi_{\l,t}$ is weakly l.s.c., by \eqref{eqn:subgr_2} and dominated convergence we obtain
\begin{align}\label{eqn:subgr_3}
   \int_0^T \Psi_t \Phi_{\l,t}(X) dt
   &\le \int_0^T \Psi_t \Phi_{\l,t}(Z) dt \nonumber\\
    &\hskip15pt + \frac{1}{2} \int_0^T \Psi_t \bigg(  \E \|X_0\|_H^2 + \E \int_0^t e^{-2\l r}\|B_r\|_{L_2(U,H)}^2  dr - \E e^{-2\l t}\|X_t\|_H^2 \bigg) dt \\
    &\hskip15pt - \limsup_{n \to \infty}  \int_0^T \Psi_t \E\int_0^t e^{-2\l r} \Vbk{ D \vp_\l(X^n_r),Z_r} dr dt. \nonumber
\end{align}
 
Application of It\^o's formula to \eqref{eqn:X-def} and the function $\|\cdot\|_H^2$ yields
\begin{align}\label{eqn:Ito_for_limit}
  &\E e^{-2\l t}\|X_t\|_H^2 = \E \|X_0\|_H^2  + \E \int_0^t e^{- 2\l r} \left( 2 \big( -(\eta_r+\l X_r),X_r \big)_H + \|B_r\|_{L_2(U,H)}^2 \right) dr.
\end{align}
Equation \eqref{eqn:subgr_3} becomes
\begin{align}\label{eqn:subgr_6}
   \int_0^T \Psi_t \Phi_{\l,t}(X) dt
   &\le \int_0^T \Psi_t \Phi_{\l,t}(Z) dt + \int_0^T \Psi_t \E \int_0^t e^{-2 \l r}  (\eta_r+\l X_r,X_r)_H dr dt\\
    &\hskip15pt - \limsup_{n \to \infty} \int_0^T \Psi_t \E\int_0^t e^{-2 \l r} \Vbk{ D \vp_\l(X^n_r),Z_r} dr dt  .\nonumber
\end{align}
Define $Z^m := \mcP_{\vp,m} Z$. By (A1), $\vp_\l(Z^m) \le C(1+\td\vp_1(Z)) \le C(1+\vp_\l(Z)) \in L^1([0,T] \times \O).$ Hence, $Z^m \in \mcD(\Phi_{\l,T})$ and \eqref{eqn:subgr_6} can be applied. Recall $P_nD\vp(X^n) \rightharpoonup \eta$ in $L^2([0,T]\times\O;H)$ and hence also in $L^2([0,t]\times\O;H)$. \eqref{eqn:subgr_6} applied to $Z = Z^m$ yields
\begin{align*}
   &\int_0^T \Psi_t \Phi_{\l,t}(X) dt
   \le \int_0^T \Psi_t \Phi_{\l,t}(Z^m) dt + \int_0^T \Psi_t \E \int_0^t  e^{-2\l r} (\eta_r + \l X_r,X_r-Z^m_r)_H dr\ dt.
\end{align*}
Since $\Psi \in L^\infty([0,T])$ was arbitrary nonnegative, consequently,
\begin{align}\label{eqn:subgr_4}
    \Phi_{\l,T}(X)\le \Phi_{\l,T}(Z^m) + \E \int_0^T  e^{-2\l r} (\eta_r +\l X_r,X_r-Z^m_r)_H dr.
\end{align}
We now want to consider the limit $m \to \infty$. By Remark \ref{rmk:vp1_dense} and dominated convergence we have $\vp(Z^m) \to \vp(Z)$ in $L^1([0,T]\times\O)$ and $Z^m \to Z$ in $L^2([0,T]\times\O;H)$. Thus, also $\Phi_{\l,T}(Z^m) \to \Phi_{\l,T}(Z)$. Passing to the limit in \eqref{eqn:subgr_4} we arrive at 
\begin{align}\label{eqn:subgr_5}
    \Phi_{\l,T}(X) \le \Phi_{\l,T}(Z) + \E \int_0^T  e^{-2\l r}(\eta_r+\l X_r,X_r-Z_r)_H dr,
\end{align}
i.e.\ $e^{-2\l \cdot} (\eta_\cdot+\l X_\cdot) \in \partial\Phi_{\l,T}(X)$. The proof of
  \[ \partial\Phi_{\l,T}(X) = \{ e^{-2\l \cdot} Y | Y \in L^2([0,T]\times \O;H) \text{ and } Y_t(\o) \in \partial \vp_\l(X_t(\o)),\ \text{dt}\otimes\P\text{-a.s.} \}\]
is similar to \cite[Proposition 2.9., p.75]{B93}. Therefore, $\eta(t,\o)+\l X_t(\o) \in \partial \vp_\l(X_t(\o))$ and since $\partial \vp(X_t(\o))$ is single-valued we conclude  $\eta(t,\o) = \partial \vp(X_t(\o))$, $\text{dt}\otimes\P$-almost surely. 
Thus
  \[X_t = X_0 - \int_0^t \partial\vp(X_r) dr + \int_0^t B_r dW_r,\ \forall t \in [0,T], \]
$\P$-almost surely.

Next we will prove \eqref{eqn:add_energy_bound}. Since $P_n D\vp(X_r^n) \rightharpoonup \partial\vp(X)$ in $L^2([0,T]\times\O;H)$:
\begin{align}\label{eqn:add_enery_bound_1}
  \E\int_0^T \|\partial\vp(X_r)\|_H^2 dr 
  &\le \liminf_{n \to \infty} \E\int_0^T \|P_nD\vp(X^n_r)\|_H^2 dr \\
  &\le C \left(\E \td\vp_1(X_0) + 1 + \E \int_0^T \|B_r\|_{L_{2,\td\vp_1},(p_i)} dr \right), \nonumber
\end{align}
where $C \in \R$ is the constant occurring in \eqref{eqn:a-priori-bdd}. For $Z \in L^2([0,T]\times\O;H)$ and $p > 1$ let $\Phi_{\l,p}(Z) := \int_0^T |\E \vp_\l(Z_t) |^p dt.$ $\Phi_{\l,p}$ is a convex, l.s.c.\ function, as $\vp_\l$ is. By Lemma \ref{lem:a-priori-bdd} we obtain
\begin{align*}
  \|\E\vp_\l(X_\cdot)\|_{L^p([0,T])} 
  &= \Phi_{\l,p}(X)^\frac{1}{p} \le \liminf_{n \to \infty} \Phi_{\l,p}(X^n)^\frac{1}{p} \le T^\frac{1}{p} \sup_{t \in [0,T]} \E \vp_\l(X^n_t)\\
  &\le C \left(\E \vp_1(X_0) + 1 + \E \int_0^T \|B_r\|_{L_{2,\td\vp_1},(p_i)} dr \right),
\end{align*}
for some $C \in \R$. Since the right hand side is independent of $p$, $X$ is continuous in $H$ and $\E\vp_\l(\cdot): L^2(\O;H) \to \bar\R$ is l.s.c., this implies
  \[ \sup_{t \in [0,T]} (\E\vp(X_t) + \|X_t\|_H^2) \le C \left(\E \vp_1(X_0) + 1 + \E \int_0^T \|B_r\|_{L_{2,\td\vp_1},(p_i)} dr \right),\]
for some $C \in \R$. Together with \eqref{eqn:add_enery_bound_1} this yields \eqref{eqn:add_energy_bound}.

Next, we prove that $X \in L^2(\O;C([0,T];H))$. Using It\^o's formula for $X_t^n$ and $\|\cdot\|_H^2$, applying Burkholder's inequality to the real-valued local martingale $\sum_{k=1}^n \int_0^t 2(X_r^n,\mcP_{\vp,n} B_r(\td e_k))_H d\b^k_r$ and using Lemma \ref{lem:a-priori-bdd}, we conclude
\begin{align*}
  \frac{1}{2} \E \sup_{t \in [0,T]}\|X_t^n\|_H^2 
  &\le \E\|X_0^n\|_H^2 + \E \int_0^T C(f_r+\|X_r^n\|^2_H) dr \le C (\E \td\vp_1(X_0) + 1).
\end{align*}
Since $X^n \rightharpoonup X$ in $L^2([0,T]\times\O;H)$, this implies $X^n \rightharpoonup X$ in $L^2(\O;L^p([0,T];H))$ with $\|X\|_{L^2(\O;L^p([0,T];H))}\le C (\E \td\vp_1(X_0) + 1)$, for every $p > 2$. By continuity of $X$ in $H$ this yields
  \[ \|X\|_{L^2(\O;C([0,T];H))}\le C (\E \td\vp_1(X_0) + 1).\]

Inequality \eqref{eqn:dep_ic_noise_add} is shown by first applying It\^o's formula to $X_t^{(1)} - X_t^{(2)}$ and $\|\cdot\|_H^2$, then using Burkholder's inequality for the real-valued local martingale $\int_0^t (X_r^{(1)}-X_r^{(2)},(B_r^{(1)}-B_r^{(2)})dW_r)_H$.
\end{proof}

\subsection{General noise}\label{subsec:general_noise}

By using a fixed point argument we will now generalize the above result to the case of not necessarily additive noise. 

Let $X_0 \in L^2(\O,\mcF_0;H)$ with $\E \td\vp_1(X_0) < \infty$, $\mcH := \{X \in L^\infty([0,T];L^2(\O;H))|\ X \text{ predictable}\}$, $\l > 0$ such that $\vp_\l$ is convex, $\mcG := \{X \in \mcH| \vp_\l(X) \in L^1([0,T] \times \O)\}$ and  $\mcG_\tau := \{X \in \mcH| \Phi_{\l,T}(X) := \E\int_0^T e^{-Cr} \vp_\l(X_r)dr \le \tau \}$, where $\tau \in \R$. Moreover, let $B: [0,T] \times \O \times V \to L_2(U,H)$ as in Theorem \ref{thm:main}. $\Phi_{\l,T}: L^2([0,T]\times\O;H) \to \bar\R$ is l.s.c., since $\vp_\l$ is l.s.c. and nonnegative. Thus, also $\Phi_{\l,T}: L^\infty([0,T];L^2(\O;H)) \to \bar\R$ is l.s.c. and thereby $\mcG_\tau$ is a closed subset of $\mcH$. For $Y \in \mcG$ we have $\|B_t(Y_t)\|_{L_{2,\td\vp_1},(p_i)} \le C(f_t+\vp(Y_t)+\|Y_t\|_H^2) \in L^1([0,T]\times\O)$.
Hence, we can apply Lemma \ref{lemma:main_add} to obtain a unique strong solution $X=F(Y)$ to \eqref{eqn:add_SPDE} with $B_t = B_t(Y_t)$, where $F: \mcG \to \mcG$ denotes the solution mapping. For $Y \in \mcG_\tau$, by Lemma \ref{lemma:main_add}
\begin{align*}
  \Phi_{\l,T}(F(Y)) 
  &= \E\int_0^T e^{-Cr} \vp_\l(X_r)dr \le T \sup_{r \in [0,T]} \E  e^{-Cr} \vp_\l(X_r)dr \\
  &\le TC \left( \E \td\vp_1(X_0) + 1 + C\E \int_0^T e^{-Cr} f_r dr + C \E \int_0^T e^{-Cr}\vp_\l(Y_r) dr \right) \\
  &\le TC \left( \E \td\vp_1(X_0) + 1 + \tau \right).
\end{align*}
For $\tau \ge 1 + \E \td\vp_1(X_0)$ we obtain
 \[\Phi_{\l,T}(F(Y)) \le (2CT) \tau.\]
Hence, $F: \mcG_\tau \to \mcG_\tau$ if $T \le \frac{1}{2C}$. Next, we prove that $F: \mcG_\tau \to \mcG_\tau$ is a contraction for $T$ small enough. Let $Y^{(1)},\ Y^{(2)} \in \mcG_\tau$. By Lemma \ref{lemma:main_add} and (A5), we have
\begin{align*}
  \|F(Y^{(1)}) - F(Y^{(2)})\|_\mcH^2 
  &\le \E \sup_{t \in [0,T]} \|F(Y^{(1)})_t - F(Y^{(2)})_t\|_H^2 \\
  &\le C\E \int_0^T \|B_r(Y^{(1)})-B_r(Y^{(2)})\|_{L_2(U,H)}^2 dr \\
  &\le CT \|Y^{(1)}-Y^{(2)}\|_\mcH^2 dr,
\end{align*}
for some constant $C \in \R$. By Banach's fixed point Theorem there exists a unique $X \in \mcG_\tau$ such that $F(X)=X$, i.e. $X$ is a strong solution of \eqref{eqn:SPDE}. Since the choice of $T$ is independent of $\E \td\vp_1(X_0)$, a standard continuation argument yields the existence of a strong solution to \eqref{eqn:SPDE} for all times. Inequality \eqref{eqn:dep_ic} immediately follows from Lemma \ref{lemma:main_add}. This finishes the proof of Theorem \ref{thm:main}.

\section{Limit solutions}\label{sec:limit_soln}

\subsection{Proof of Theorem \ref{thm:limit_soln}}

  Let $X_0 \in L^2(\O,\mcF_0;H)$ and define $G_n := \{ h \in H | \td\vp_1(h) \le n \}$. Since $\td\vp_1$ is convex and l.s.c., $G_n$ is a convex, closed set in $H$. Therefore, the $\|\cdot\|_H$-best-approximation map $P_n: H \to G_n$ is well defined by
    \[ \|h - P_nh\|_H = \inf_{g \in G_n} \|h-g\|_H \]
 and is continuous (cf. \cite{V73}, \cite[Proposition 2.4]{K60}). Set $X_0^n := P_n X_0$. Hence, $\E \td\vp_1(X_0^n) \le n$ and there exists a unique strong solution $X^n$ to \eqref{eqn:SPDE} with initial condition $X_0^n$. Moreover, $\|X_0^n\|_H \le 2\|X_0\|_H$ and $X_0^n(\o) \to X_0(\o)$, for each $\o \in \O$, since the embedding $\bigcup_{n \in \N} G_n = V \subseteq H$ is dense. Thus, $X_0^n \to X_0$ in $L^2(\O;H)$. By Theorem \ref{thm:main} we conclude
    \[ \E \sup_{t \in [0,T]} \|X_t^n-X_t^m\|_H^2 \le C\E\|X_0^n-X_0^m\|_H^2 \to 0,\]
  for $n,m \to \infty$, i.e. $X_t^n$ is a Cauchy sequence in $L^2(\O;C([0,T];H))$ and we obtain a limit solution $X^n \to X \in L^2(\O;C([0,T];H))$. Inequality \eqref{eqn:limit_soln_ctn} obviously follows from \eqref{eqn:dep_ic}.

\subsection{Regularity of limit solutions (proof of Theorem \ref{thm:regularity})}

For two strong solutions $X^{(i)}$, $i=1,2$ to \eqref{eqn:SPDE} on $(0,T]$ with the same initial condition, we can apply It\^o's formula to $\|X^{(1)}-X^{(2)}\|_H^2$ on $[\d,T]$ and use $X \in C([0,T];L^2(\O;H))$ to prove the uniqueness of strong solutions to \eqref{eqn:SPDE} on $(0,T]$.

\begin{lemma}\label{lemma:add_r12_bound}
  Assume (A1)-(A6) and (A4'), $X_0 \in L^2(\O,\mcF_0;H)$ with $\E\vp_1(X_0) < \infty$ and that $B_t(v) \equiv B_t$ is independent of $v\in V$. Denote by $X$ the strong solution to \eqref{eqn:add_SPDE}. Then
    \[ \E \int_0^T \|r^\frac{1}{2}\partial\vp(X_r)\|_H^2 dr \le C \left( \E \|X_0\|_H^2 + 1 + \E\int_0^T  \|B_r\|_{L_{2,\td\vp_1},(p_i)} dr \right), \]
  for some constant $C\in\R$.
\end{lemma}
\begin{proof}
  We first consider the approximating solutions $X^n$ to \eqref{eqn:approx_eqn}. It\^o's formula and (A4') yield
  \begin{align*}
    &\E e^{-cT}\|X_T^n\|_H^2 \\
    &\hskip15pt \le \E \|X_0^n\|_H^2 + \E \int_0^T e^{-cr} \bigg( C_2(1+\|X_r^n\|_H^2) - C_1 \vp(X_r^n) + C\|B_r\|_{L_{2,\td\vp_1},(p_i)} - c \|X_r^n\|_H^2 \bigg)dr,
  \end{align*}
  for any $c > 0$. Hence
  \begin{equation}\label{eqn:approx_uniform_vp_bound}
    \E \int_0^T  \vp(X_r^n)dr \le C \left( \E \|X_0^n\|_H^2 + 1 + \E\int_0^T \|B_r\|_{L_{2,\td\vp_1},(p_i)} dr \right).
  \end{equation}    
  Let $F(t,v) := t\vp(v) \in C^2(\R\times H_n)$. By It\^o's formula and proceeding as in Lemma \ref{lem:a-priori-bdd}
  \begin{align*}
     t\vp(X_t^n) 
    &= \int_0^t r D \vp(X_r^n)(-P_n D\vp(X_r^n)) dr + \sum_{k=1}^n \int_0^t r D\vp(X_r^n)(\mcP_{\vp,n}B_r(\td e_k)) d\b^k_r + \int_0^t \vp(X_r^n)dr \nonumber\\
      &\hskip15pt + \frac{1}{2} \sum_{l=1}^n \int_0^t r D^2\vp(X_r^n) (\mcP_{\vp,n}B_r(\td e_l),\mcP_{\vp,n}B_r(\td e_l))dr \\
    &\le - \int_0^t \|P_n r^\frac{1}{2} D\vp(X_r^n)\|_H^2 dr + \sum_{k=1}^n \int_0^t r D \vp(X_r^n)(\mcP_{\vp,n}B_r(\td e_k)) d\b^k_r  \nonumber\\
      &\hskip15pt + C \int_0^t \vp(X_r^n)dr + C \int_0^t (1+\|B_r(v)\|_{L_{2,\td\vp_1},(p_i)}) dr.\nonumber
  \end{align*}
  Hence, applying \eqref{eqn:approx_uniform_vp_bound} yields
  \begin{align*}
     \E \int_0^T \|P_n r^\frac{1}{2}D\vp(X_r^n)\|_H^2 dr
    &\le  C \left( \E \|X_0^n\|_H^2 + 1 + \E\int_0^T  \|B_r(v)\|_{L_{2,\td\vp_1},(p_i)} dr \right).
  \end{align*}
  
  Now we pass to the limit $n\to \infty$. Let $\mcG := L^2([0,T]\times\O;H)$ with norm $\|Z\|_\mcG^2 := \E\int_0^T r \|Z_r\|_H^2 dr$. Then $P_nD\vp(X^n)$ is bounded in $\mcG$. Hence (for a subsequence)
    \[ P_nD\vp(X^n) \rightharpoonup \td\eta\]
  in $\mcG$. For $0<\d<T$, $\|Z\|_\d^2 := \E\int_\d^T r \|Z_r\|_H^2 dr$ is an equivalent norm to $\|\cdot\|_{L^2([\d,T]\times\O;H)}$ on $L^2([\d,T]\times\O;H)$. Therefore, $P_nD\vp(X^n) \rightharpoonup \td\eta$ in $L^2([\d,T]\times\O;H)$. By the proof of Lemma \ref{lemma:main_add} we know $P_nD\vp(X^n_r) \rightharpoonup \partial\vp(X),$ in $L^2([0,T]\times\O;H)$. We conclude $\td\eta = \partial\vp(X)$, dt$\otimes\P$-almost surely. In the proof of Lemma \ref{lemma:main_add} it has been shown $X_0^n \to X_0 \in L^2(\O;H)$. By weak lower semicontinuity of $\|\cdot\|_\mcG$ we conclude
  \begin{align*}
    \E\int_0^T \|r^\frac{1}{2} \partial\vp(X_r) \|_H^2 dr 
    &\le \liminf_{n \to \infty} \E\int_0^T \|P_n r^\frac{1}{2} D\vp(X^n_r)\|_H^2 dr \\
    &\le C \left( \E \|X_0\|_H^2 + 1 + \E\int_0^t  \|B_r\|_{L_{2,\td\vp_1},(p_i)} dr \right),
  \end{align*}
  for some $C \in \R$.
\end{proof}

\begin{lemma}\label{lemma:12_strong}
  Assume (A1)-(A6) and (A4'). Let $X_0 \in L^2(\O,\mcF_0;H)$ such that $\E\vp_1(X_0) < \infty$ and let $X$ be the corresponding strong solution. Then
    \[ \E \int_0^T \|r^\frac{1}{2}\partial\vp(X_r)\|_H^2 dr \le C \left( \E \|X_0\|_H^2 + 1 + \E\int_0^T f_r dr \right), \]
  for some constant $C \in \R$.
\end{lemma}
\begin{proof}
   Applying It\^o's formula for $\|\cdot\|_H^2$ to $X$ and using (A4') we obtain
   \begin{align}\label{eqn:Ito_A4'}
      \E e^{-cT}\|X_T\|_H^2 
      = \E \|X_0\|_H^2 - \E \int_0^T e^{-cr} (C_1 \vp(X_r) + (c-C)\|X_r\|_H^2 ) dr + \int_0^T e^{-Cr} C dr ,
  \end{align}
  for any $c > 0$ and some $C \in \R$ large enough. Let $\mcH$, $\mcG$, $F: \mcG \to \mcG$ as in Section \ref{subsec:general_noise} and recall that $X$ is constructed as a fixed point of the map $F: \mcG \to \mcG$, mapping $Y \in \mcG$ to the unique solution of \eqref{eqn:add_SPDE} with $B_r = B_r(Y_r)$. By Lemma \ref{lemma:add_r12_bound}, \eqref{eqn:Ito_A4'} and (A6) we have
  \begin{align*}
    \E \int_0^T \|r^\frac{1}{2}\partial\vp(X_r)\|_H^2 dr 
    &\le C \left( \E \|X_0\|_H^2 + 1 + \E\int_0^T \|B_r(X_r)\|_{L_{2,\td\vp_1},(p_i)} dr \right) \\
    &\le C \left( \E \|X_0\|_H^2 + 1 + \E\int_0^T C(f_r + \vp_1(X_r))dr \right) \\
    &\le C \left( \E \|X_0\|_H^2 + 1 + \E\int_0^T f_r dr \right),
  \end{align*}
  for some constant $C \in \R$.
\end{proof}

\begin{proof}[proof of Theorem \ref{thm:regularity}]
  Let $\l > 0$ such that $\vp_\l$ is convex. Since $\vp_\l$ is l.s.c., $\E\int_0^T \vp_\l(\cdot)dt: L^2([0,T]\times\O;H) \to \bar\R$ is convex and weakly l.s.c.. Let $X^n$ denote an approximating sequence of strong solutions for $X$. By \eqref{eqn:Ito_A4'} we have
   \[ \E\int_0^T \vp_\l(X_r)dr \le \liminf_{n \to \infty} \E\int_0^T \vp_\l(X^n_r)dr \le C(\E\|X_0\|_H^2+1). \]
  Let $0 < \delta < T$. Since $\vp_\l(X) \in L^1([0,T]\times\O)$ there is a $\d_0 \in [0,\delta]$ such that $\E \vp_1(X_{\d_0}) < \infty$. Hence, there exists a unique strong solution $X^{(\d_0)}$ to 
  \begin{align}\label{eqn:SPDE_start_d0}
     X_t = X_{\d_0} - \int_{\d_0}^t \partial\vp(X_r) dr + \int_{\d_0}^t B_r(X_r) dW_r.
  \end{align}
  Since every strong solution is a limit solution, the uniqueness of limit solutions implies $X^{(\d_0)} = X, \text{ on } [\d_0,T]$. By Lemma \ref{lemma:12_strong} and Fatou's Lemma we obtain
  \begin{align*}
     \E \int_0^T \|r^\frac{1}{2}\partial\vp(X_r)\|_H^2 dr 
     \le \liminf_{\d_0 \to 0} \E \int_{\d_0}^T \|r^\frac{1}{2}\partial\vp(X_r)\|_H^2 dr 
      \le C \left( \E \|X_{0}\|_H^2 + 1 + \E\int_0^T f_r dr \right).
  \end{align*}
  In particular $\partial\vp(X) \in L^2([\d,T] \times \O;H)$ for all $\d > 0$ and $X$ is a strong solution of \eqref{eqn:SPDE} on $(0,T]$.
\end{proof}

We now comment on the proof of Remark \ref{rmk:limit=varn}. Given a limit solution $X^{(l)}$ we consider the corresponding approximating sequence $X^{(l),n}$ of strong solutions. The additional coercivity and growth assumptions imply enough regularity for $X^{(l),n}$ to apply It\^o's formula (\cite[Theorem 4.2]{RRW07}) for the difference $X^{(l),n}-X^{(v)}$ and the function $\|\cdot\|_H^2$. Taking the limit $n\to \infty$ then yields the assertion.

\section{Applications}\label{sec:applications}
 
We now apply our results to stochastic generalized porous media equations, stochastic generalized reaction diffusion equations and stochastic generalized $p$-Laplace equations, i.e.\ we specify conditions on the coefficients of the drift so that (A1)-(A4) are satisfied and in case of additive or linear multiplicative noise we identify concrete conditions on the coefficients of the noise to satisfy (A5), (A6).

\subsection{Stochastic Generalized Porous Media Equations}\label{sec:SPME}
We consider stochastic generalized porous media equations (SPME) based on the setup established in \cite[Section 3]{RRW07}. Let $(E,\mcB,m)$ be a $\s$-finite measure space with countably generated $\s$-algebra $\mcB$ and let $(L,\mcD(L))$ be a negative-definite, self-adjoint operator on $L^2(m)$ with Ker$(L)=\{0\}$. We consider SPDE of the form
\begin{equation}\label{eqn:SPME}
  dX_t = L \Phi(X_t) dt + B_t(X_t)dW_t,
\end{equation}
where $\Phi \in C^1(\R), \Phi(0) = 0$ with
\begin{align}\label{eqn:Phi_assumptions}
  0 \le \Phi'(r) &\le c_2(\mathbbm{1}_{m(E)<\infty}+|r|^{p-1}),\\
  - c_2 \mathbbm{1}_{m(E)<\infty}+c_1|r|^{p+1} & \le \Psi(r) := \int_0^r \Phi(s)ds,  \nonumber
\end{align}
for some $p \ge 1$ and some constants $c_1 > 0$, $c_2 \in \R$. In particular, the standard SPME on a not necessarily bounded domain will be included in our general framework. 

For comparison, we will now recall existing regularity results for the standard SPME, i.e.\ for the case $L = \D$, $\Phi(r) = |r|^{p-1}r$, $E = \mcO \subseteq \R^d$ open, bounded with smooth boundary $\partial \mcO$. As mentioned in the introduction, the abstract result on invariance of subspaces for SPDE given in \cite{L10} applied to the standard SPME yields $L^2(\mcO)$ invariance, i.e.\ if $X_0 \in L^2(\O;L^2(\mcO))$ and $B \in L^2([0,T]\times\O;L_2(L^2(\mcO)))$ then the variational solution $X$ to the standard SPME is a right continuous process in $L^2(\mcO)$. The same invariance result, but for a more general class of SPME had been obtained in \cite{RW08} by similar methods proving also $|X|^\frac{p-1}{2}X \in L^2([0,T]\times\O;H_0^1(\mcO))$ in case of the standard SPME. For additive noise the same spatial regularity had been proved in \cite{K06} by entirely different methods. For more regular noise, namely $B: H \to L_2(L^2(\mcO);D((-\D)^\g))$ for some $\g > \frac{d}{2}$ it has been shown that $\int_0^\cdot \Phi(X_r)$ is $\P$-a.s.\ weakly continuous in $H_0^1(\mcO)$ for a large class of nonlinearities $\Phi$ \cite{RW08}. To the author's knowledge, highest known regularity for the solution of the standard SPME with additive noise has been obtained in \cite{BGLR10} proving $|X|^{p-1}X \in L^\frac{p+1}{p}([0,T];H_0^{1,\frac{p+1}{p}}(\mcO))$. Our results strengthen these regularity results to 
  \[ |X|^{p-1}X \in L^2([0,T]\times\O;H_0^1(\mcO)).\]
  
We denote $m(fg) := \int_E fg dm$ for $fg \in L^1(m)$, $(f,g) := (f,g)_{L^2(m)}$ and $\|f\| := \|f\|_{L^2(m)}$. Define $\mcD(\mcE) := \mcD(\sqrt{-L})$ and $\mcE(u,v) := (\sqrt{-L}u,\sqrt{-L}v), \text{ for } u,v \in \mcD(\mcE).$ Let $\mcF_e$ denote the abstract completion of $\mcD(\mcE)$ with respect to $\|\cdot\|_{\mcF_e}^2 := \mcE(\cdot,\cdot)$ and choose
  \[ H := \mcF_e^*. \]
From now on we assume
\begin{packed_enum}
 \item[(L1)] There exists a strictly positive function $g \in L^1(m) \cap L^\infty(m)$ such that $\mcF_e \subseteq L^1(g \cdot m)$ continuously.
\end{packed_enum}
Define
  \[ S := L^{p+1}(m), \]
which is a separable, reflexive Banach space. We now informally define $V = S \cap H$, made precise in the following sense:
  \[ V := \{ u \in S|\ \exists c >0 \text{ such that } m(uv) \le c \|v\|_{\mcF_e},\ \forall v \in \mcF_e \cap S^*  \}, \]
with norm $\|\cdot\|_V := \|\cdot\|_S + \|\cdot\|_H.$ Assume
\begin{packed_enum}
 \item[(N1)] $\mcF_e \cap S^* \subseteq \mcF_e$ dense. 
 \item[(N2)] $V \subseteq H,\ S$ dense.
\end{packed_enum}
As shown in \cite{RRW07} (N1) implies $V \subseteq H$ continuously and that $V$ is complete. Since $V$ is isomorphic to a closed subspace of $S \times H$, $V$ is reflexive. We recall 
\begin{proposition}[\cite{RRW07}]
  Let $(\mcE,\mcD(\mcE))$ be a transient Dirichlet space. Then (L1), (N1), (N2) hold.
\end{proposition}

\begin{example}
  A lot of examples of transient Dirichlet spaces, for example with $E$ being a manifold are given in the literature. We now mention several examples for $E = \mcO \subseteq \R^d$ being an open set and $m$ being the Lebesgue measure:
  \begin{packed_enum}
    \item Let $E$ bounded and $L$ be the Friedrichs extension of a symmetric, uniformly elliptic operator of second order with Dirichlet boundary conditions. 
    \item Let $E = \R^d_+ = \{ \xi \in \R^d|\ \xi_1 \ge 0\}$, $L=\D$ with Neumann boundary conditions, $d \ge 2$.
    \item Let $E = \R^d$, $L=(-\D)^\a$ with its standard domain and $\a \in (0,\frac{d}{2}) \cap (0,1]$.
  \end{packed_enum}
  In all these cases, $(\mcD(\mcE),\mcE)$ defines a transient Dirichlet space and thus (L1), (N1), (N2) are satisfied. For details and a proof we refer to \cite{FOT94}.
\end{example}

For $x \in S$ we define
\begin{align*}
  \vp(x) := \int_E \Psi(x(\xi)) dm(\xi), \qquad  \td\vp := \frac{1}{p+1} \|\cdot\|_S^{p+1}.
\end{align*}
(A1) is obviously satisfied. We have $\vp, \td\vp \in C^2(S)$ with
\begin{align*}
   & D\vp(x)(h) = \int_E \Phi(x) h dm, \qquad   D^2\vp(x)(h,g) = \int_E \Phi'(x) hg dm.
\end{align*}
We observe:
\begin{align*}
  \sum_{k=1}^\infty D^2\vp(x)(w_k,w_k) 
  &= \sum_{k=1}^\infty \int_E \Phi'(x) |w_k|^2 dm \\
  &\le C \sum_{k=1}^\infty \int_E (\mathbbm{1}_{m(E)<\infty}+|x|^{p-1})|w_k|^2 dm\\
  &=  \left( \mathbbm{1}_{m(E)<\infty}+\frac{1}{p+1}\|x\|_{p+1}^{p+1} \right) + C \left(\sum_{k=1}^\infty\|w_k\|_{p+1}^2 \right)^\frac{p+1}{2} \\
  &\le  C \left(1+ \vp(x) +\sum_{k=1}^\infty \td\vp(w_k)^\frac{2}{p+1} \right)^\frac{p+1}{2}.
\end{align*}
Hence, (A2) is satisfied with $p_1 = \frac{p+1}{2}$. Obviously $\vp: H \to \bar\R$ is convex, i.e.\ (A3) is satisfied. For (A4), (A4') note that $\Phi(r)r \ge \int_0^r \Phi(s) ds \ge -c_2 \mathbbm{1}_{m(E)<\infty} + c_1|r|^{p+1}$ and
\begin{align*}
  \Vbk{-D\vp(v),v} = -\int_E \Phi(v)v dm \le \int_E c_2 \mathbbm{1}_{m(E) < \infty} - c_1 |v|^{p+1}  dm \le C_2 - C_1 \td\vp(v).
\end{align*}
We obtain:

\begin{theorem}\label{thm:SPME}
  Assume (L1), (N1), (N2). Let $X_0 \in L^2(\O,\mcF_0;H)$ and $B:[0,T] \times\O \times H \to L_2(U,H)$ be progressively measurable and satisfy (A5), (A6). Then there exists a unique strong solution $X \in L^2(\O;C([0,T];H))$ on $(0,T]$ with $X(0)=X_0$ to \eqref{eqn:SPME} with
  \begin{align*}
    \E \int_0^T \|X_t\|_{p+1}^{p+1} + t\|\Phi(X_t)\|_{\mcF_e}^2 dt < \infty.
  \end{align*}    
  The variational solution to \eqref{eqn:SPME} as obtained in \cite[Theorem 3.9]{RRW07} coincides with the strong solution obtained here. If additionally $\E \|X_0\|_{p+1}^{p+1} < \infty$, then $X$ is a strong solution of \eqref{eqn:SPME} with
  \begin{align*}
    &\sup_{t \in [0,T]} \E \|X_t\|_{p+1}^{p+1} + \E \int_0^T \|\Phi(X_t)\|_{\mcF_e}^2 dt< \infty.
  \end{align*}    
\end{theorem}

\begin{remark}
  Theorem \ref{thm:SPME} can easily be generalized to SPDE of the form
    \[ dX_t = L\left( \sum_{i=1}^N \Phi_i(X_t) \right)dt + B_t(X_t)dW_t, \]
  where $\Phi_i \in C^1(\R)$ are as in \eqref{eqn:Phi_assumptions} with $p = p_i$.
\end{remark}

\begin{remark}
  In the variational setup \cite{RRW07} applied to the standard SPME with Dirichlet boundary conditions, it is not clear in which sense the variational solution satisfies the boundary conditions since these are not explicitly incorporated in the space $V$ nor in $H$. A partial solution to this problem has been suggested in \cite{RW08} by proving $\int_0^t \Phi(X_r) dr \in H_0^1(\mcO)$ almost surely. By Theorem \ref{thm:SPME} the variational solution coincides with the strong solution on $(0,T]$ and thus is shown to satisfy the boundary conditions in the strong sense that $\Phi(X_t) \in H_0^1(\mcO)$ almost surely. 
\end{remark}

We now give explicit examples of noise satisfying (A5), (A6). In \cite{K06} the standard porous medium equation perturbed by additive Wiener noise is considered, i.e. $B(u) := \sum_{k=1}^\infty g_k (\td e_k,u)_U.$ The unique existence of a variational solution $X \in L^2(\O;L^\infty([0,T];L^2(\mcO)))$ satisfying $|X|^\frac{p-2}{2}X \in L^2([0,T]\times\O;H_0^1(\mcO))$ is proven under the assumptions $X_0 \in L^2(\O;L^2(\mcO))$, $g_k \in L^2([0,T]\times\O\times\mcO)$ progressively measurable and either $\mcO \subseteq \R^d$ open, bounded with smooth boundary and $\E \int_0^T \sum_{k=1}^\infty \|g_k\|_2^2 < \infty$ or $\mcO = \R^d$ and $\E \left( \int_0^T \sum_{k=1}^\infty \|g_k\|_2^2 \right)^\frac{p}{2} < \infty$.

\begin{remark}
  Assume $g_k \in L^2([0,T]\times\O;H)$ progressively measurable and 
    \[ \E \int_0^T \left( \sum_{k=1}^\infty \|g_k\|_{p+1}^2 + \|g_k\|_H^\frac{4}{p+1} \right)^\frac{p+1}{2} < \infty.\]
  Then $B(u) := \sum_{k=1}^\infty g_k (\td e_k,u)_U$ satisfies (A5), (A6).
\end{remark}

For stochastic porous media equations, the case of random linear multiplicative space-time perturbations has been intensively studied in the literature (cf.\ e.g.\ \cite{BDPR08,bdpr2009}). In the setup $E = \mcO \subseteq \R^d$ open, bounded with smooth boundary, $d \le 3$, $L=\D$ on $L^2(\mcO)$ with Dirichlet boundary conditions, $\Phi \in C^1(\R)$ as above, $B(v)(h) := \sum_{k=1}^\infty \mu_k \td e_k v (\td e_k,h)_{L^2(\mcO)}$, where $\mu_k \in \R$, $\td e_k$ an orthonormal basis of eigenvectors of $-\D$ on $L^2(\mcO)$, $\l_k$ the corresponding eigenvalues, $S = V = L^{p+1}(\mcO)$ and $H = (H_0^1(\mcO))^*$, the unique existence of variational solutions for initial values $x \in H$ has been proven in \cite{BDPR08} under the assumption $\sum_{k=1}^\infty \l_k^2 \mu_k^2 < \infty.$ For the sake of simplicity and comparability we will now restrict to this setup. Using $\|\td e_k\|_\infty \le C \l_k$ and $\|\td e_k v\|_H^2 \le C \l_k^2 \|x\|_H^2$ (cf. \cite[p. 189]{BDPR08}) it follows

\begin{remark}
   Assume 
      \[ \sum_{k=1}^\infty \mu_k^\frac{4}{p+1} \l_k^\frac{4}{p+1} < \infty. \]
   Then $B(v)(h) := \sum_{k=1}^\infty \mu_k \td e_k v (\td e_k,h)_{L^2(\mcO)}$ satisfies (A5), (A6).
\end{remark}

\subsection{Stochastic Generalized Reaction Diffusion Equations}\label{sec:SRDE}
We will now consider stochastic generalized reaction diffusion equations of the form
\begin{equation}\label{eqn:SRDE}
  dX_t = \left(L X_t + \sum_{i=1}^N f_i(X_t) \right)dt + B_t(X_t)dW_t,
\end{equation}
where the operator $L$ will be specified below and $f_i \in C^1(\R)$, with
\begin{equation}\label{eqn:reaction_cdts}
  \setlength\arraycolsep{0.1em}
 \begin{array}{rclclcl}
  f'_i(t) &\le& c_2,        &\qquad&     |f'_i(t)| &\le& c_2(\mathbbm{1}_{m(E)<\infty}+|t|^{r_i-2}) \\
  c_2 (-\mathbbm{1}_{m(E)<\infty} - |t|^{r_i}) &\le& F_i(t),  &\qquad&     F_i(t), f_i(t)t &\le& -c_1|t|^{r_i} + c_2\mathbbm{1}_{m(E)<\infty},
 \end{array} 
\end{equation}
for some constants $c_1,c_2 > 0$, $r_N \ge ... \ge r_1 \ge 2$, $F_i'=f_i$ and all $t \in \R$. 

For example, let $m(E) < \infty$ (w.l.o.g.\ $N=1$) and $f_1: \R \to \R$ be a polynomial of odd degree with negative leading coefficient, i.e.\ $f_1(t) = \sum_{i=1}^{n} a_i t^{p_i},$ $1 \le p_1 \le ... \le p_n$, $p_n$ odd and $a_{n} < 0$. Then \eqref{eqn:reaction_cdts} is satisfied with $r_1 := p_n+1 \ge 2$. For $m(E) = \infty$ we can consider reaction terms given by a polynomial $ f(t) = \sum_{i=1}^{N} a_i t^{p_i} $ with odd orders $1 \le p_1 \le ... \le p_N$ and with negative coefficients $a_i < 0$, by letting $f_i(t) = a_i t^{p_i}$.

Let $(E,\mcB,m)$ be as in Section \ref{sec:SPME} and $(L,\mcD(L))$ be a non-negative, self-adjoint operator on $L^2(m)$ with corresponding closed, symmetric form $(\mcE,\mcD(\mcE))$. Recall that $\mcF_1 := \mcD(\mcE)$ is a separable Hilbert space with respect to the inner product $\mcE_1(g,h) := \mcE(g,h) + (g,h)$. Define
\begin{align*}
  H &= L^2(m), \\
  S &= V = \mcF_1 \cap L^{r_1}(m) \cap ... \cap L^{r_N}(m).
\end{align*}
Completeness of $V$ follows from the embedding of $V$ into the space of all equivalence classes of $\mcB / \mcB(\R)$-measurable functions $L^0(E,\mcB,m)$ with the topology of convergence in measure. Hence, $V$ is isomorphic to a closed subspace of the reflexive space $\mcF_1 \times L^{r_1}(m) \times ... \times L^{r_N}(m)$ and thus is reflexive. We assume
\begin{packed_enum}
 \item[(N1')] $V \subseteq H$ dense.
\end{packed_enum}
Set
  \[ \vp(v) := \frac{1}{2} \mcE(v,v) - \sum_{i=1}^N \int_E F_i(v) dm, \qquad \td\vp(v) := \frac{1}{2} \mcE(v,v) + \sum_{i=1}^N \int_E |v|^{r_i} dm, \]
for $v \in V$. (A1) immediately follows since the mappings $v \mapsto \int_E F_i(v) dm,\ \int_E |v|^{r_i} dm$ are in $C^2(L^{r_i}(m))$ and by using \eqref{eqn:reaction_cdts}. (A3) follows from $F_i'' \le c_2$. For (A2) we note $\vp ,\td\vp \in C^2(S)$, with
  \begin{align*}
    D\vp(v)(w_1)        &= \mcE(v,w_1) - \sum_{i=1}^N \int_E f_i(v)w_1 dm \\
    D^2\vp(v)(w_1,w_2)  &= \mcE(w_1,w_2) - \sum_{i=1}^N \int_E f_i'(v) w_1 w_2 dm.
  \end{align*}
  We observe
  \begin{align}\label{eqn:reaction_C^2_bound}
    \sum_{k=1}^\infty \int_E f_i'(v) w_k w_k dm 
    &\le \|f_i'(v)\|_{\frac{r_i}{r_i-2}} \left(\sum_{k=1}^\infty\|w_k^2\|_\frac{r_i}{2} \right) \\
    &\le C \left( \|f_i'(v)\|_{\frac{r_i}{r_i-2}}^{\frac{r_i}{r_i-2}} + \left(\sum_{k=1}^\infty \|w_k^2\|_\frac{r_i}{2} \right)^\frac{r_i}{2} \right) \nonumber\\
    &\le C \left( 1+ \td\vp(v) + \left(\sum_{k=1}^\infty \td\vp(w_k)^\frac{2}{r_i} \right)^\frac{r_i}{2} \right),\nonumber
  \end{align}
  for some constant $C > 0$. Hence
  \begin{align*}
      \sum_{k=1}^\infty D^2\vp(x)(w_k,w_k) 
      &\le \sum_{k=1}^\infty \mcE(w_k,w_k) + \sum_{k=1}^\infty \sum_{i=1}^N \int_E f_i'(v) w_k w_k dm \\
      &\le \sum_{k=1}^\infty \mcE(w_k,w_k) + C  \sum_{i=1}^N \left( 1+ \td\vp(v) + \left(\sum_{k=1}^\infty \td\vp(w_k)^\frac{2}{r_i} \right)^\frac{r_i}{2} \right) \\
      &\le C  \left( 1+ \td\vp(v) +  \sum_{i=0}^N\left(\sum_{k=1}^\infty \td\vp(w_k)^\frac{1}{p_i} \right)^{p_i} \right),
  \end{align*}
  for some constant $C > 0$, i.e. (A2) is satisfied with $p_i = \frac{r_i}{2}$. (A4) follows from \eqref{eqn:reaction_cdts}.

  \begin{theorem}\label{thm:SRDE}
     Assume (N1'). Let $X_0 \in L^2(\O,\mcF_0;H)$ and $B:[0,T] \times\O \times H \to L_2(U,H)$ be progressively measurable and satisfy (A5), (A6). Then there exists a unique strong solution $X \in L^2(\O;C([0,T];H))$ on $(0,T]$ with $X(0)=X_0$ to \eqref{eqn:SRDE} with
     \begin{align*}
        &\E \int_0^T \Big( \|X_t\|_{\mcF_1}^2 + \sum_{i=1}^N \|X_t\|_{r_i}^{r_i} \Big) + t \Big( \|LX_t + \sum_{i=1}^N f_i(X_t)\|_2^2 \Big) dt < \infty.
     \end{align*}

     If $\E\left( \|X_0\|_{\mcF_1}^2 + \sum_{i=1}^N \|X_0\|_{r_i}^{r_i}\right) < \infty$,  then $X_t$ is the unique strong solution with
     \begin{align*}
       \sup_{t \in [0,T]} \E \left( \|X_t\|_{\mcF_1}^2 + \sum_{i=1}^N\|X_t\|_{r_i}^{r_i} \right) + \E\int_0^T \|LX_t + \sum_{i=1}^N f_i(X_t)\|_2^2 dt  &< \infty.
     \end{align*}
  \end{theorem}

In the special case $E= \mcO \subseteq \R^n$ open, bounded, $L=\D$ on $L^2(\mcO)$ with Dirichlet boundary conditions we can assume without loss of generality $N=1$ and define $f:=f_1$ with $r:=r_1$.
\begin{remark}
  Let $X_0 \in L^2(\O,\mcF_0;H)$ and $B:[0,T] \times\O \times H \to L_2(U,H)$ be progressively measurable and satisfy (A5), (A6). Let $X_t$ be the unique strong solution on $(0,T]$ of
  \[ dX_t = \left( \D X_t + f(X_t)\right) dt + B_t(X_t)dW_t.\]
  Then
  \begin{align*}
     &\E \int_0^T \|X_t\|_{H_0^1}^2 + \|X_t\|_{r}^{r} + t \left( \|X_t\|_{H^2}^2 + \|X_t\|_{2(r-1)}^{2(r-1)} \right) dt < \infty. 
  \end{align*}
  and if $\E \left( \|X_0\|_{H_0^1}^2 + \|X_0\|_{r}^{r} \right) < \infty$ then $X$ is a strong solution with
  \begin{align*}
    \sup_{t \in [0,T]} \E \left( \|X_t\|_{H_0^1}^2 + \|X_t\|_r^r \right) + \E\int_0^T \|X_t\|_{H^2}^2 + \|X_t\|_{2(r-1)}^{2(r-1)} dt  &< \infty.
  \end{align*}
\end{remark}

\begin{remark}
  Extending our results to additive inhomogeneities is immediate. Applied to equations of the form
    \[ dX_t = (\D X_t + f(t)) dt + B_t dW_t \]
  our results then yield $X \in L^2([0,T]\times\O;H^2(\mcO)) \cap L^\infty([0,T];L^2(\O;H^1_0(\mcO)))$ which has also been shown in \cite[Theorem 3.1]{F90} under the same compatibility relations used here, i.e. $X_0 \in H_0^1(\mcO)$ and $B_t(\td e_k) \in L^2([0,T]\times\O;H_0^1(\mcO))$. Note that in \cite[Proposition 3.3]{F90} these compatibility conditions have been shown to be necessary at least for more regular solutions. 
\end{remark}

For ease of notation and comparability to existing results, we now restrict to the case $E = \mcO \subseteq \R^n$ open, bounded, $L = \D$ on $L^2(\mcO)$ with Dirichlet boundary conditions.

\begin{remark}
 \begin{packed_enum}
   \item Assume $g_k \in L^2([0,T]\times\O \times \mcO)$ progressively measurable and 
	 \[ \E \int_0^T \left( \sum_{k=1}^\infty \|g_k\|_{H_0^1}^\frac{4}{r_N} + \|g_k\|_{r_N}^\frac{2r_1}{r_N} \right)^\frac{r_N}{2} < \infty.\]
	 Then $B_t(u) := \sum_{k=1}^\infty g_k (\td e_k,u)_U$ satisfies (A5), (A6).
   \item Assume
	 \[ \sum_{k=1}^\infty \mu_k^{\frac{4}{r_N} \wedge \frac{2r_1}{r_N}} \l_k^{\frac{4}{r_N} \wedge \frac{2r_1}{r_N}} < \infty,\]
	 where $\mu_k \in \R$, $\td e_k \in H_0^{1,\infty}(\mcO)$ is an orthonormal basis of $L^2(\mcO)$ and $\l_k := \|\td e_k\|_\infty + \|\nabla \td e_k\|_\infty$. Then $B(v)(h) := \sum_{k=1}^\infty \mu_k \td e_k v (\td e_k,h)_{L^2(\mcO)}$ satisfies (A5), (A6).
 \end{packed_enum}
\end{remark}

\subsection{Stochastic Generalized Degenerate $p$-Laplace Equations}

We now consider a class of stochastic generalized degenerated $p$-Laplace equations with reaction term of polynomial growth, thus partially generalizing the results from Section \ref{sec:SRDE}. More explicitly, we consider equations of the form
\begin{equation}\label{eqn:SPLE}
  dX_t = \left( div(\Phi(\nabla X_t)) + \sum_{i=1}^n f_i(X_t) \right) dt + B_t(X_t)dW_t,
\end{equation}
on an open set $\mcO \subseteq \R^d$ with smooth boundary $\partial \mcO$, $f_i \in C^1(\R)$ as in Section \ref{sec:SRDE} with $(E,\mcB,m) = (\mcO,\mcB(\mcO),dx)$ and with one of the following boundary conditions:
\begin{packed_enum}
  \item Dirichlet boundary conditions and $\mcO$ such that Poincar\'e's inequality holds (cf.\ \cite[Theorem 6.28]{A75}).
  \item Neumann boundary conditions and $\mcO$ bounded and convex.
\end{packed_enum}
In case of Neumann boundary conditions, if $m \le r_n$ the convexity of $\mcO$ can be dropped and if $m \in [r_1,r_n]$ also boundedness of $\mcO$ is not needed. We assume, $\Phi = \nabla \Psi: \R^d \to \R^d$ for some convex function $\Psi \in C^2(\R^d;\R)$ satisfying
\begin{align*}
  &\Phi(0) = \nabla \Psi(0) = 0, \\
  &\|D^2\Psi(x)\|_{\R^d\times\R^d} \le c_2(\mathbbm{1}_{|\mcO|<\infty}+|x|^{m-2}), \\
  &-  c_2 \mathbbm{1}_{|\mcO|<\infty} + c_1 |x|^m \le \Psi(x),
\end{align*}
for some $m \ge 2$ and some constants $c_1, c_2 > 0$. This implies $-\Phi(x) \cdot x \le -c_1 |x|^m + c_2\mathbbm{1}_{|\mcO|<\infty}$ and $\Psi(x) \le c_2(\mathbbm{1}_{|\mcO|<\infty}+|x|^m)$, for some constants $c_1, c_2 > 0$. 

In the special case of the standard $p$-Laplace equation (i.e. $\Phi(r) = |r|^{m-2}r$) with low order reaction terms (i.e. $r_i \le m$), Neumann boundary conditions on a bounded, convex set $\mcO$, $B \in L^2([0,T]\times\O;L_2(L^2(\mcO);H^1(\mcO)))$ and $X_0 \in L^2(\O;H^1(\mcO))$, the abstract invariance result obtained in \cite{L10} can be applied to prove that $X_t$ is a right-continuous process in $H^1(\mcO)$ with $\E \sup_{t \in [0,T]} \|X_t\|_{H^1(\mcO)}^2 < \infty$. Applied to the standard $p$-Laplace equation with Neumann boundary conditions for $X_0 \in L^2(\O;H)$ we (partially) strengthen this to 
  \[ \E\int_0^T t\|div (|\nabla X_t|^{m-2} \nabla X_t)\|_2^2dt < \infty,\]
and for $X_0 \in L^m(\O;W^{1,m}(\mcO))$ we get
  \[\sup_{t \in [0,T]} \E \|X_t\|_{W^{1,m}(\mcO)}^m + \E\int_0^T \|div (|\nabla X_t|^{m-2} \nabla X_t)\|_2^2dt < \infty.\]
In the following, we will consider the case of Dirichlet boundary conditions. Neumann boundary conditions can be treated analogously. Let $H = L^2(\mcO)$, $S = W_0^{1,m}(\mcO) \cap L^{r_1}(\mcO) \cap ... \cap L^{r_n}(\mcO)$ and $V = S \cap H$. Let $F_i \in C^2(\R)$ such that $F_i' = f_i$ and set
  \[ \vp(v) := \int_\mcO \Psi(\nabla v) d\xi - \sum_{i=1}^n \int_\mcO F_i(v) d\xi, \quad \td\vp(v) := \frac{1}{m} \int_\mcO |\nabla v|^m d\xi + \sum_{i=1}^n \int_\mcO |v|^{r_i} d\xi, \]
for $v \in S$. We have $\vp,\td\vp \in C^2(S)$ with
\begin{align*}
  D\vp(v)(w) &= \int_\mcO \Phi(\nabla v)(\nabla w) d\xi - \sum_{i=1}^n \int_\mcO f_i(v)w d\xi, \\
  D^2\vp(v)(w_1,w_2) &= \int_\mcO D^2\Psi(\nabla v)(\nabla w_1, \nabla w_2) d\xi - \sum_{i=1}^n \int_\mcO f_i'(v)w_1w_2 d\xi.
\end{align*}
Using this, convexity of $\Psi$ and \eqref{eqn:reaction_cdts}, (A1), (A3), (A4) and (A4') are easily verified. For (A2) we consider the case $m > 2$ (for $m=2$ proceed similarly). First, note
    \begin{align*}
      \sum_{k=1}^\infty \int_\mcO D^2\Psi(\nabla v)(\nabla w_k, \nabla w_k) d\xi 
      &\le \|D^2\Psi(\nabla v)\|_{\frac{m}{m-2}} \sum_{k=1}^\infty \| |\nabla w_k|^2\|_{\frac{m}{2}} \\
      &\le C \left( 1+\|\nabla v\|_m^m + \left( \sum_{k=1}^\infty \|\nabla w_k\|_m^{2} \right)^{\frac{m}{2}} \right) \\
      &\le C \left( 1+\vp(v) + \left( \sum_{k=1}^\infty \td\vp(w_k)^{\frac{2}{m}} \right)^{\frac{m}{2}} \right),
    \end{align*}
  for some constant $C > 0$. Taking into account \eqref{eqn:reaction_C^2_bound} we obtain
    \begin{align*}
        \sum_{k=1}^\infty D^2\vp(x)(w_k,w_k) &\le C \left( 1+\vp(v) + \left( \sum_{k=1}^\infty \td\vp(w_k)^{\frac{2}{m}} \right)^{\frac{m}{2}} +  \sum_{i=1}^n\left(\sum_{k=1}^\infty \td\vp(w_k)^\frac{2}{r_i} \right)^\frac{r_i}{2}\right),
    \end{align*}
    for some constant $C > 0$. Thus, (A2) is satisfied with $\{p_i\} = \{\frac{r_i}{2}\} \cup \{\frac{m}{2} \}$, $N=n+1$.

\begin{theorem}\label{thm:p-Laplace}
  Let $X_0 \in L^2(\O,\mcF_0;H)$ and $B:[0,T] \times\O \times H \to L_2(U,H)$ be progressively measurable and satisfy (A5), (A6). Then there exists a unique strong solution $X \in L^2(\O;C([0,T];H))$ on $(0,T]$ with $X(0)=X_0$ to \eqref{eqn:SPLE} with
  \begin{align*}
    &\E \int_0^T \left( \|X_t\|_{W^{1,m}_0(\mcO)}^m + \sum_{i=1}^n \|X_t\|_{r_i}^{r_i}\right) + t \left(\|div\ \Phi(\nabla X_t) + \sum_{i=1}^n f_i(X_t)\|_2^2\right) dt < \infty.
  \end{align*}
  If $\E\left( \|X_0\|_{W_0^{1,m}(\mcO)}^m + \sum_{i=1}^n \|X_0\|_{r_i}^{r_i}\right) < \infty$,  then $X_t$ is the unique strong solution with
  \begin{align*}
    \sup_{t \in [0,T]} \E \left( \|X_t\|_{W_0^{1,m}(\mcO)}^m + \sum_{i=1}^n\|X_t\|_{r_i}^{r_i} \right) + \E\int_0^T \|div\ \Phi(\nabla X_t) + \sum_{i=1}^n f_i(X_t)\|_2^2 dt  &< \infty.
  \end{align*}
\end{theorem}

\begin{remark}
  For the standard $p$-Laplace equation on a bounded set $\mcO \subseteq \R^d$ with Dirichlet boundary conditions (w.l.o.g.\ $n \le 1$) we get $t^\frac{1}{2} \|div(|\nabla X_t|^{m-2}\nabla X_t) \|_2 + t^\frac{1}{2}\|f_1(X_t)\|_2 \in L^2([0,T] \times \O)$ and if $\E\left( \|X_0\|_{W_0^{1,m}(\mcO)}^m + \|X_0\|_{r_1}^{r_1}\right) < \infty$, then $X_t$ is the unique strong solution with $\|div(|\nabla X_t|^{m-2}\nabla X_t)\|_2 + \|f_1(X_t)\|_2 \in L^2([0,T] \times \O)$.
\end{remark}
\hskip5pt
\begin{remark}
  \begin{packed_enum}
   \item Assume $g_k \in L^2([0,T]\times\O \times \mcO)$ progressively measurable and 
	    \[ \E \int_0^T \left( \sum_{k=1}^\infty \|g_k\|_{W^{1,m}_0(\mcO)}^{1 \wedge \frac{2m}{r_N}} + \|g_k\|_{r_N}^{\frac{2}{m} \wedge \frac{4}{r_N}}  \right)^{\frac{r_N}{2} \vee m} < \infty.\]
	 Then $B_t(u) := \sum_{k=1}^\infty g_k (\td e_k,u)_U$ satisfies (A5), (A6).
   \item Assume
	    \[ \sum_{k=1}^\infty (\mu_k \l_k)^{1 \wedge \frac{4}{m} \wedge \frac{4}{r_n}} < \infty, \]
          where $\mu_k \in \R$, $\td e_k \in H_0^{1,\infty}(\mcO)$ is an orthonormal basis of $L^2(\mcO)$ and $\l_k := \|\td e_k\|_\infty + \|\nabla \td e_k\|_\infty$. Then $B(v)(h) := \sum_{k=1}^\infty \mu_k \td e_k v (\td e_k,h)_{L^2(\mcO)}$ satisfies (A5), (A6).
  \end{packed_enum}
\end{remark}


\bibliographystyle{amsalpha.bst}
\bibliography{../latex-refs/refs}   

\providecommand{\bysame}{\leavevmode\hbox to3em{\hrulefill}\thinspace}
\providecommand{\MR}{\relax\ifhmode\unskip\space\fi MR }
\providecommand{\MRhref}[2]{%
  \href{http://www.ams.org/mathscinet-getitem?mr=#1}{#2}
}
\providecommand{\href}[2]{#2}
\begin{thebibliography}{BvNVW08}

\bibitem[Ada75]{A75}
Robert~A. Adams, \emph{Sobolev spaces}, Academic Press [A subsidiary of
  Harcourt Brace Jovanovich, Publishers], New York-London, 1975, Pure and
  Applied Mathematics, Vol. 65.

\bibitem[Bar93]{B93}
Viorel Barbu, \emph{Analysis and control of nonlinear infinite-dimensional
  systems}, Mathematics in Science and Engineering, vol. 190, Academic Press
  Inc., Boston, MA, 1993.

\bibitem[BDPR08]{BDPR08}
Viorel Barbu, Giuseppe Da~Prato, and Michael R{\"o}ckner, \emph{Existence and
  uniqueness of nonnegative solutions to the stochastic porous media equation},
  Indiana Univ. Math. J. \textbf{57} (2008), no.~1, 187--211.

\bibitem[BDPR09]{bdpr2009}
\bysame, \emph{Stochastic porous media equations and self-organized
  criticality}, Comm. Math. Phys. \textbf{285} (2009), no.~3, 901--923.

\bibitem[BGLR11]{BGLR10}
Wolf-J\"urgen Beyn, Benjamin Gess, Paul Lescot, and Michael R{\"o}ckner,
  \emph{The global random attractor for a class of stochastic porous media
  equations}, Comm. Partial Differential Equations \textbf{36} (2011), no.~3,
  446--469.

\bibitem[BvN03]{BvN03}
Zdzislaw Brze{\'z}niak and Jan M. A.~M. van Neerven, \emph{Space-time
  regularity for linear stochastic evolution equations driven by spatially
  homogeneous noise}, J. Math. Kyoto Univ. \textbf{43} (2003), no.~2, 261--303.

\bibitem[BvNVW08]{BvNVW08}
Zdzislaw Brze{\'z}niak, Jan M. A.~M. van Neerven, Mark~C. Veraar, and Lutz
  Weis, \emph{It\^o's formula in {UMD} {B}anach spaces and regularity of
  solutions of the {Z}akai equation}, J. Differential Equations \textbf{245}
  (2008), no.~1, 30--58.

\bibitem[DPZ92]{DPZ92}
Giuseppe Da~Prato and Jerzy Zabczyk, \emph{Stochastic equations in infinite
  dimensions}, Encyclopedia of Mathematics and its Applications, vol.~44,
  Cambridge University Press, Cambridge, 1992.

\bibitem[Fla90]{F90}
Franco Flandoli, \emph{Dirichlet boundary value problem for stochastic
  parabolic equations: compatibility relations and regularity of solutions},
  Stochastics Stochastics Rep. \textbf{29} (1990), no.~3, 331--357.

\bibitem[Fla95]{F95}
\bysame, \emph{Regularity theory and stochastic flows for parabolic {SPDE}s},
  Stochastics Monographs, vol.~9, Gordon and Breach Science Publishers,
  Yverdon, 1995.

\bibitem[F{\=O}T94]{FOT94}
Masatoshi Fukushima, Y{\=o}ichi {\=O}shima, and Masayoshi Takeda,
  \emph{Dirichlet forms and symmetric {M}arkov processes}, de Gruyter Studies
  in Mathematics, vol.~19, Walter de Gruyter \& Co., Berlin, 1994. \MR{1303354
  (96f:60126)}

\bibitem[GS82]{GS82}
V.~A. Gnatyuk and V.~S. Shchirba, \emph{General properties of best
  approximation with respect to a convex continuous function}, Ukrain. Mat. Zh.
  \textbf{34} (1982), no.~5, 608--613, 677.

\bibitem[JR10]{JR10}
Arnulf Jentzen and Michael R\"ockner, \emph{Regularity analysis for stochastic
  differential equations with nonlinear multiplicative trace class noise},
  preprint (2010).

\bibitem[Kim06]{K06}
Jong~U. Kim, \emph{On the stochastic porous medium equation}, J. Differential
  Equations \textbf{220} (2006), no.~1, 163--194.

\bibitem[Kle61]{K60}
Victor Klee, \emph{Convexity of {C}hevyshev sets}, Math. Ann. \textbf{142}
  (1960/1961), 292--304.

\bibitem[KR79]{KR79}
Nicolai~V. Krylov and Boris~L. Rozovski{\u\i}, \emph{Stochastic evolution
  equations}, Current problems in mathematics, {V}ol. 14 ({R}ussian), Akad.
  Nauk SSSR, Vsesoyuz. Inst. Nauchn. i Tekhn. Informatsii, Moscow, 1979,
  pp.~71--147, 256.

\bibitem[Kry94]{K94}
Nicolai~V. Krylov, \emph{A {$W^n_2$}-theory of the {D}irichlet problem for
  {SPDE}s in general smooth domains}, Probab. Theory Related Fields \textbf{98}
  (1994), no.~3, 389--421.

\bibitem[Liu10]{L10}
Wei Liu, \emph{Invariance of subspaces under the solution flow of {SPDE}},
  Infin. Dimens. Anal. Quantum Probab. Relat. Top. \textbf{13} (2010), no.~1,
  87--98.

\bibitem[PR07]{PR07}
Claudia Pr{\'e}v{\^o}t and Michael R{\"o}ckner, \emph{A concise course on
  stochastic partial differential equations}, Lecture Notes in Mathematics,
  vol. 1905, Springer, Berlin, 2007.

\bibitem[PZ07]{PZ07}
Szymon Peszat and Jerzy Zabczyk, \emph{Stochastic partial differential
  equations with {L}\'evy noise}, Encyclopedia of Mathematics and its
  Applications, vol. 113, Cambridge University Press, Cambridge, 2007, An
  evolution equation approach.

\bibitem[Roz90]{R90}
Boris~L. Rozovski{\u\i}, \emph{Stochastic evolution systems}, Mathematics and
  its Applications (Soviet Series), vol.~35, Kluwer Academic Publishers Group,
  Dordrecht, 1990, Linear theory and applications to nonlinear filtering,
  Translated from the Russian by A. Yarkho.

\bibitem[RRW07]{RRW07}
Jiagang Ren, Michael R{\"o}ckner, and Feng-Yu Wang, \emph{Stochastic
  generalized porous media and fast diffusion equations}, J. Differential
  Equations \textbf{238} (2007), no.~1, 118--152.

\bibitem[RS81]{RS81}
Boris~L. Rozovski{\u\i} and A.~Shimizu, \emph{Smoothness of solutions of
  stochastic evolution equations and the existence of a filtering transition
  density}, Nagoya Math. J. \textbf{84} (1981), 195--208.

\bibitem[RW08]{RW08}
Michael R{\"o}ckner and Feng-Yu Wang, \emph{Non-monotone stochastic generalized
  porous media equations}, J. Differential Equations \textbf{245} (2008),
  no.~12, 3898--3935.

\bibitem[Vla73]{V73}
Leonid~P. Vlasov, \emph{Approximative properties of sets in normed linear
  spaces}, Uspehi Mat. Nauk \textbf{28} (1973), no.~6(174), 3--66.

\bibitem[Zha07]{Z07}
Xicheng Zhang, \emph{Regularities for semilinear stochastic partial
  differential equations}, J. Funct. Anal. \textbf{249} (2007), no.~2,
  454--476.

\end{thebibliography}

\end{document}